\documentclass{article}
\usepackage{amsmath}
\usepackage{bm}
\usepackage{mathrsfs}
\usepackage{times}
\usepackage{a4wide}
\usepackage{amssymb}
\usepackage{amsthm}
\usepackage{hyperref}
\numberwithin{equation}{section}
\newtheorem{theorem}{Theorem}[section]
\newtheorem{lemma}{Lemma}[section]

\newtheorem{corollary}{Corollary}[section]
\newtheorem{remark}{Remark}[section]
\newcommand{\iv}{\operatorname{div}}
\newcommand{\Lbr}{\{\!\!\{}
\newcommand{\Rbr}{\}\!\!\}}
\newcommand{\Lbrack}{[\![}
\newcommand{\Rbrack}{]\!]}

\title{An interior penalty discontinuous Galerkin method with correct and minimal averages, jumps and penalties for the miscible displacement problem of nonnegative characteristic form, and SUPG-type error estimates under low regularity, dominating Darcy velocity}
\author{Zhijie Du\thanks{School of Mathematics and Statistics, Wuhan University of Technology, Wuhan 430070, China. E-mail:zjdu@whu.edu.cn} \and Huoyuan Duan\thanks{School of Mathematics and Statistics, Wuhan University, Wuhan 430072, China.
E-mail: hyduan.math@whu.edu.cn} \and Roger C. E. Tan\thanks{Department of Mathematics, National University of Singapore, Singapore 119076, Singapore. E-mail: mattance@nus.edu.sg} \and Yuanhong Wei\thanks{School of Mathematics and Statistics, Wuhan University, Wuhan 430072, China. E-mail:yuanhongwei@whu.edu.cn}}
\date{}

\begin{document}

\maketitle

\begin{abstract}
An interior penalty discontinuous Galerkin (DG) method is proposed for the steady-state linear partial differential equations of nonnegative characteristic form or vanishing diffusion coefficient, suitable for mixed second-order elliptic-parabolic and first-order hyperbolic equations. Due to the different natures of the elliptic, parabolic, and hyperbolic equations, the solution must have the discontinuity, and the DG method is preferable in the finite element discretization. In the new DG method, the averages, jumps and penalties are minimal, correctly and only imposed on the diffusion-diffusion element boundaries, in addition to the well-known upwind jumps associating with the advection velocity. For the advection-dominated problem, the penalties can be further reduced only being imposed on the diffusion-dominated subset of the diffusion-diffusion element boundaries. \textit{This is based on the novel, crucial technique, which we have developed for the purpose of the theoretical analysis,  about the multiple partitions of the set of the interelement boundaries into a number of subsets with respect to the diffusion and to the advection and on the consistency result we have proven.}  In the literature, to the best knowledge of the authors, all DG methods introduce unnecessary and unphysical averages, jumps and penalties in the presence of the vanishing diffusion coefficient and of the dominating advection velocity, e.g., the physical discontinuity of the solution is also artificially penalized. The new DG method is the first DG method and the first time that the continuity and discontinuity of the solution are correctly identified and justified of the general steady-state linear partial differential equations of nonnegative characteristic form.   

The new DG method and its analysis are applied to the miscible displacement problem of vanishing diffusion coefficient and of low regularity, dominating Darcy flow velocity (advection velocity) which lives in $H(\iv;\Omega)\cap \prod_{j=1}^J (H^r(D_j))^d$ for $r<1$ 
other than the usual assumption $(W^{1,\infty}(\Omega))^d$ or $(H^2(\Omega))^d$ or $(W^{2,4}(\Omega))^d$. We prove the SUPG-type error estimates $\mathcal{O}(h^{\ell+\frac{1}{2}})$ for any element polynomial of degree $\ell\ge 1$ on generally shaped and nonconforming meshes, where the convergence order is independent of the regularity $r$ or $s$ of the advection velocity. \textit{The novel, crucial technique for the error estimates is the introduction of a key parameter $\mathcal{D}_K(A,\mathbf{u},\gamma,h_K)$ about the diffusion $A$, the advection velocity $\mathbf{u}$, the reaction coefficient $\gamma$ and the element diameter $h_K$.} The SUPG-type error estimates obtained are new and the first time known under the low regularity of the advection velocity.   
\\

\\

\end{abstract}

\maketitle 

\section{Introduction}
\label{sec:Introduction}

In this paper, we shall be concerned with the discontinuous Galerkin (DG) finite element method for the general steady-state linear partial differential equations of nonnegative characteristic form   
\begin{equation}\label{eq:second-order pde}
\iv(-A\nabla c+\mathbf{u}c)+\gamma c=g,
\end{equation}
where the symmetric matrix diffusion coefficient $A\in\mathbb{R}^{d\times d}$ may vanish somewhere in $\Omega\subset\mathbb{R}^d$, $d=2,3$, satisfying 
\begin{equation}\label{eq:vanish diffusion}
\xi^\top A(\mathbf{x})\xi\ge 0 \quad\forall \xi\in\mathbb{R}^d, \text{ a.e. } \mathbf{x}\in\bar{\Omega}.
\end{equation} 
Such vanishing diffusion coefficient renders the name partial differential equations of nonnegative characteristic form to \eqref{eq:second-order pde}, cf. \cite{HoustonSuli2001}, \cite{HoustonSchwabSuli2002}. The partial differential equations \eqref{eq:second-order pde} presents the different natures of second-order elliptic-parabolic problems and first-order hyperbolic problem in some subregions of $\Omega$, and consequently, there must have some discontinuity-interfaces on which the solution has physical intrinsic discontinuities. Because of \eqref{eq:vanish diffusion}, the advection velocity $\mathbf{u}$ and the reaction coefficient $\gamma$ must dominate over the diffusion coefficient in their magnitudes, in general. As is well-known, the advection-dominated problems have been interesting in the finite element methods.     

The discontinuous Galerkin method is preferable simultaneously discretizing the advection-dominated problem and the discontinuous solution, and there have been numerous literature available. It also allows essentially arbitrarily-shaped elements (e.g., polytopal elements) and arbitrarily nonconforming triangulations (cf. \cite{CangianiDongGeorgoulis2022}). The general properties of the DG methods are averages, jumps and penalties, where the artificial penalties are imposed on the jumps across the interelement boundaries and play a decisive role in the stability, which may read as follows: 
$
\sum_{\Lambda\in\mathcal{F}_h} \alpha_\Lambda h_\Lambda^{-1} \int_\Lambda \Lbrack c_h\Rbrack\Lbrack v_h\Rbrack
$. 
The DG methods generally penalize the jumps on all the interelement boundaries, as is well-known. There is no problem whenever \eqref{eq:second-order pde} is a diffusive elliptic equation satisfying 
\begin{equation}\label{eq:elliptic diffusion}
\xi^\top A(\mathbf{x})\xi>0 \quad\forall \xi\in\mathbb{R}^d, \text{ a.e. } \mathbf{x}\in\bar{\Omega}.
\end{equation}  
The penalization mainly accounts  for the discontinuity in the finite element space while the solution often has the regularity $H^1(\Omega)$. However, for the general vanishing diffusion coefficient \eqref{eq:vanish diffusion}, penalizing all jumps is not necessary; for the limit case $A\equiv 0$, i.e., the first-order hyperbolic problem, as is well-known, no artificial penalties are introduced. There are a few DG methods (cf. \cite{HoustonSchwabSuli2002}, \cite{ProftRiviere2009} and references therein)  
which dealt with the vanishing diffusion coefficient, but unnecessary jumps are penalized over the interelement boundaries and unphysical continuities are assumed over some subsets of the domain $\Omega$ where the solution is actually discontinuous. Such issue in the DG method was firstly noticed in \cite{HoustonSchwabSuli2002}, but was not further studied. As is well-known in the pure diffusion problem, roughly speaking, the role of the penalties is to restore the continuity of the continuous solution because of the discontinuity caused by the discontinuous finite element space. However, as pointed out in \cite{HoustonSchwabSuli2002}, it is quite unnatural to penalize those discontinuous jumps of the physical discontinuity of the solution whose discontinuity happens at the places where the solution changes its nature from the elliptic equation to parabolic equation and/or to hyperbolic equation. It turns out from our study that the penalizations need to be imposed only on the diffusion-diffusion interelement boundaries and even only on the diffusion-dominated subset of the diffusion-diffusion interelement boundaries. 

Our study, as a matter of act, is based on a novel, crucial technique we have developed. It is about the multiple partitions of the set of the interelement boundaries of $\mathcal{F}_h^0$ (the set of interior element boundaries) into two parts (see \eqref{eq:diffusion element boundary-0}): one part is the subset of the diffusion-diffusion interelement boundaries $\mathcal{F}_h^{\text{df-df},0}$, see \eqref{eq:diffusion element boundary}, for which \eqref{eq:elliptic diffusion} is satisfied respectively with $\xi:=\mathbf{n}_{\partial K_1}$ and with $\xi:=\mathbf{n}_{\partial K_2}$ ($\mathbf{n}$ is the unit normal to the interelement boundary $\Lambda=\partial K_1\cap \partial K_2\in \mathcal{F}_h^{\text{df-df},0}$) on both sides of $\Lambda$, while the rest is the subset of the advection interelement boundaries $\mathcal{F}_h^{\text{ad},0}$, see  \eqref{eq:advection element boundary}, for which \eqref{eq:elliptic diffusion} is satisfied on at most one side of the interelement boundary $\Lambda\in\mathcal{F}_h^0$. Then, we further partition the subset  $\mathcal{F}_h^{\text{ad},0}$ into two subsubsets: one subsubset is called the advection-diffusion set $\mathcal{F}_h^{\text{ad-df},0}$ (see \eqref{eq:ad-df set}), the other is called the advection-advection set $\mathcal{F}_h^{\text{ad-ad},0}$ (see \eqref{eq:ad-ad set}); for the latter there are no diffusions on both sides of the interelement boundaries. Afterwards, to correctly identify what parts of the advection-diffusion set $\mathcal{F}_h^{\text{ad-df},0}$ need to impose the continuity of the solution, we partition this set once more time into two parts: the inflow part $\mathcal{F}_h^{\text{ad-df},0,-}$ and the outflow part $\mathcal{F}_h^{\text{ad-df},0,+}$, see \eqref{eq:ad-df-inflow-outflow}.  These multiple partitions, laying the foundation of our DG method, are used for the theoretical analysis only.   

From the consistency result we have rigorously shown of our DG method, it is revealed that the continuity conditions of the solution across the interior element boundaries are correctly identified and justified according to the multiple partitions of the set of the interelement boundaries. Such identifications 
read as follows (see Theorem \ref{thm:consistency}): 
\begin{equation}\label{eq:jum continuity}
\begin{aligned}
& \Lbrack c\Rbrack|_{\mathcal{F}_h^{\text{df-df},0}}=0,
\\
&\mathbf{u}\cdot\mathbf{n}\Lbrack c\Rbrack|_{\mathcal{F}_h^{\text{ad-df},0,+}}=0,
\\
& \text{No continuities are imposed on } {\mathcal{F}_h^{\text{ad-df},0,-}},
\\
& \mathbf{u}\cdot\mathbf{n}\Lbrack c\Rbrack|_{\mathcal{F}_h^{\text{ad-ad},0}}=0.
\end{aligned}
\end{equation}  
The above facts correctly give the requirements on the continuity of the solution, which are new, unknown in the literature, although the first continuity is well-known in the pure diffusion problem (i.e., $\mathbf{u}\equiv 0,\gamma\equiv 0$, and \eqref{eq:elliptic diffusion} holds) and the last continuity is well--known in the pure advection-reaction problem (i.e., $A\equiv 0$).  From \eqref{eq:jum continuity}, the solution needs to be continuous only across the diffusion-diffusion element boundaries $\mathcal{F}_h^{\text{df-df},0}$, while elsewhere, the solution itself is allowed to be discontinuous. Therefore, the above continuities comply with the realistic problems.  An example in \cite[Example 4, page 2159]{HoustonSchwabSuli2002}, which seems not to be correctly dealt with by any known DG methods, well illustrates the correctness of \eqref{eq:jum continuity}, see Remark \ref{rem:example}. It is based on the above correct continuities that our DG method imposes the correct, minimal averages, jumps, and penalties: these averages and penalties only live on the diffusion-diffusion element boundaries and the jumps only live on the diffusion-diffusion element boundaries and on the advection-advection element boundaries (the latter are the classical upwind jumps along the normal direction of the advection velocity). We need only to penalize the jumps on the diffusion-diffusion element boundaries, corresponding to the first continuity condition of \eqref{eq:jum continuity}. 

Moreover, the penalties can be further reduced for the advection-dominated problems. This is based on a further partition of the set  $\mathcal{F}_h^{\text{df-df},0}$ of the diffusion-diffusion interelement boundaries into two parts: one part is 
$\mathcal{F}_h^{\text{\rm df-df-dfd},0}$ of the diffusion-dominated diffusion-diffusion interelement boundaries, see \eqref{eq:df-df-dfd set},  for which the diffusion dominates over the advection on the interelement boundary $\Lambda\in \mathcal{F}_h^{\text{df-df},0}$, while the rest is $\mathcal{F}_h^{\text{\rm df-df-add},0}$, see \eqref{eq:advection-dominated df-df-0},  for which the advection dominates over the diffusion on the interelement boundary $\Lambda\in \mathcal{F}_h^{\text{df-df},0}$. On the diffusion Dirichlet boundary condition portion of $\Gamma$, we do similar partitions, see \eqref{eq:df-dfd-D set} and \eqref{eq:advection-dominated df-df-1}. Then, the jumps need only to be penalized on $\mathcal{F}_h^{\text{\rm df-df-dfd},0}$ the diffusion-dominated subset of the diffusion-diffusion interelement boundaries, while no penalties are needed on $\mathcal{F}_h^{\text{\rm df-df-add},0}$ the advection-dominated subset of the diffusion-diffusion interelement boundaries.   

As far as we know, our DG method is new, the first DG method and the first time that correctly and minimally imposes the averages, jumps and penalties for the steady-state, linear partial differential equations of the nonnegative characteristic form. The novel and crucial technique is the multiple partitions of the interelement boundaries with respect to the diffusion and to the advection, see \eqref{eq:correct boundary subset}. The continuity requirements are rigorously identified through the multiple partitions of the interelement boundaries and are rigorously justified from the consistency result we have shown. We must note that the multiple partitions are mainly used for the theoretical analysis (see Remark \ref{rem:multiple partition}) and do not enter into the DG method.    

In addition to the dominating property over the diffusion coefficient, the advection velocity may have a low regularity. This is the case in the incompressible miscible displacement problem. This problem has been playing an important role in hydrology and petroleum reservoir simulations (cf. \cite{Feng1995}, \cite{ChenEwing1999}), which reads as follows: in a domain $\Omega\subset\mathbb{R}^d$ ($d=2,3$) and in a time interval $(0,T)$, for the source functions $f$ and $g$,  
\begin{equation}\label{eq:diffusion-conveciton-reaction}
\phi c_t-\operatorname{div}(\mathbb{D}(\mathbf{u}) \nabla c-\mathbf{u} c)=g \quad \text { in } \quad \Omega \times(0, T), 
\end{equation}
where $c$ is the solvent concentration, and the Darcy flow velocity $\mathbf{u}$ satisfies, with the fluid pressure $p$,  
\begin{equation}\label{eq:Darcy}
\begin{aligned}
& \mathbf{u}=-\mathbb{K}(c) \nabla p\quad\mbox{in $\Omega \times(0, T)$}, 
\\
&\operatorname{div}(\mathbf{u})=f \quad \mbox{in $\Omega \times(0, T)$}.
\end{aligned}
\end{equation}
In the above system, $\phi$ is the porosity, $\mathbb{K}\in\mathbb{R}^{d\times d}$ depends nonlinearly on $c$, which describes the permeability of the porous medium, and $\mathbb{D}(\mathbf{u})\in\mathbb{R}^{d\times d}$ is the diffusion-dispersion coefficient.  There are initial and interfacial conditions and boundary conditions which are needed to supplement the above system.  

In this paper, we are mainly concerned about the finite element discretization of the concentration equation \eqref{eq:diffusion-conveciton-reaction} in its linearized, steady-state, where the Darcy velocity $\mathbf{u}$ has a low regularity and dominates over the diffusion-dispersion coefficient and the diffusion-dispersion coefficient may vanish in some subregions in the sense of \eqref{eq:vanish diffusion}. In realistic world, on one hand, since $\mathbb{K}$ is discontinuous, anisotropic, inhomogeneous and $\Omega$ has a nontrivial topology, the Darcy velocity $\mathbf{u}$ is usually singular, belonging to 
\begin{equation}\label{eq:low regularity}
H(\operatorname{div};\Omega)\cap \mbox{$\prod_{j=1}^J (H^r(D_j))^d$}\quad\mbox{for a partition $\Omega=\cup_{j=1}^J D_j$ and for a real number $0<r<1$};
\end{equation}
on the other hand, the Darcy velocity $\mathbf{u}$ may dominate over the diffusion-dispersion coefficient by a large magnitude, and \eqref{eq:diffusion-conveciton-reaction} becomes a type of advection-dominated problem. In addition, the diffusion-dispersion coefficient may not be uniformly elliptic, i.e., it only holds \eqref{eq:vanish diffusion}. 

For an advection-dominated problem, as is well-known, it can be efficiently solved by the SUPG (Streamline Upwind Petrov Galerkin) method and various stabilized methods, cf. \cite{BrooksHughes1982}. The SUPG-type stability and the SUPG-type error estimates are two advantages of the stabilized finite element methods. The SUPG-type stability can be formulated in the seminorm in terms of the advection velocity (here is the Darcy velocity $\mathbf{u}$),  like the one as follows: 
$\sqrt{\sum_{K\in\mathcal{T}_h} \tau_K \|\mathbf{u}\cdot\nabla v_h\|_{L^2(K)}^2}$, 
where 
$\tau_K \approx h_K$, whenever the advection is dominating over the diffusion, is called the stabilizing parameter.  The SUPG-type error estimates is the error bound 
\begin{equation}\label{eq:SUPG error}
O(h^{\ell+\frac{1}{2}})
\end{equation}
in the case of the dominating advection, where $\ell$ is the order of approximation of the finite element space of the solution. In the literature of the DG methods, as far as we known, in order to obtain the SUPG-type stability and the SUPG-type error bound, it has been assumed that the advection velocity $\mathbf{u}$ 
\begin{equation}\label{eq:smooth convection}
\mathbf{u}\in (W^{\ell,\infty}(\Omega))^d,\quad \ell\ge 1. 
\end{equation}
Such assumption obviously is not practical in the above system  \eqref{eq:diffusion-conveciton-reaction}- \eqref{eq:Darcy}. 

For a low regularity $\mathbf{u}$, an appropriate finite element discretization is the $H(\operatorname{div};\Omega)$-conforming elements (e.g., Raviart-Thomas elements) with a favorable conservation property, cf. \cite{BartelsJensenMuller2009}, \cite{LiSun2013}, \cite{Wang2008}, \cite{SunKou2015}, \cite{SunWu2021}, \cite{GaoSun2020}, \cite{SunYuan2009}, \cite{WangSiSun2014}, \cite{Sun2021}. Then, the resultant advection velocity, which is used in a follow-up step in the discretization of \eqref{eq:diffusion-conveciton-reaction}, cannot satisfy the assumption \eqref{eq:smooth convection}, in general. 

However, the low regularity $\mathbf{u}$ in the sense of \eqref{eq:low regularity} is not dealt with in most of the DG methods, still assuming \eqref{eq:smooth convection} (cf. \cite{HoustonSchwabSuli2002}, \cite{AyusoMarini2009},  \cite{DiPietroDroniouErn2015}, \cite{BuffaHughesSangalli2006}, \cite{DiPietroErnGuermond2008}). For low regularity solutions, we were aware of only a few literature (\cite{BartelsJensenMuller2009}, \cite{LiRiviereWalkington2015}, \cite{GiraultLiRiviere2016}, \cite{DroniouEymardPrignetTalbot2019}, \cite{RiviereWalkington2011}, \cite{SunWu2021}). They studied a general convergence. Some important issues are still missing therein, though. It is not clear whether these methods hold the advantageous SUPG-type error estimates under the low regularity Darcy velocity in the sense of \eqref{eq:low regularity}. As a matter of fact, we have not been aware of any DG method which  obtained the SUPG-type error estimates for the low regularity advection velocity (the Darcy velocity).  What is more,  few literature is known for the vanishing diffusion-dispersion coefficient. We also mention the very recent studies from the classical mixed finite element method in \cite{GaoSun2020}, \cite{SunWu2021}, \cite{Sun2021}. Therein very interesting error estimates results are obtained under the assumption that the Darcy velocity (advection velocity) $\mathbf{u}$ belongs to $(H^2(\Omega))^d$ or $(W^{2,4}(\Omega))^d$. Generally speaking, to well approximate the advection-dominated problem, the DG method, more or less, must improve the stability in the spirit of the SUPG-type stability and must reach the SUPG-type error bound. But, all DG methods without an artificial stabilization like the SUPG-type stabilized finite element method cannot provide SUPG-type stability, unless \eqref{eq:smooth convection} holds or unless $\mathbf{u}$ is some lowest-order piecewise-polynomial on the same mesh as the concentration. For the miscible displacement problem, however, the last two conditions are actually difficult to be fulfilled, since $\mathbf{u}$ generally has the low regularity \eqref{eq:low regularity} and is solved in a mesh totally different than that of the concentration (because $\mathbf{u}$ is discretized with the $H(\iv;\Omega)$-conforming element). 

In this paper, we apply and analyze our new interior DG method for the general linear partial differential equations of nonnegative characteristic form to numerically solve the incompressible miscible displacement problem \eqref{eq:diffusion-conveciton-reaction}, of which a steady-state and linearized model \eqref{eq:second-order pde} is studied, in dealing with the vanishing diffusion in the sense of \eqref{eq:vanish diffusion} and with the low regularity advection velocity in the sense of \eqref{eq:low regularity}.  In order to obtain the optimal SUPG-type error estimates in the sense of \eqref{eq:SUPG error}, we augment the DG formulation with a residual-based SUPG-type stabilization so that a SUPG-type stability can hold for a varying, non-piecewise-polynomial, low regularity advection velocity of \eqref{eq:low regularity}. 
Such stabilization seems to be indispensable, as we have not been aware of any DG method itself could generally provide SUPG-type stability without similar stabilizations for a non-piecewise-polynomial, low regularity \eqref{eq:low regularity} 
of $\mathbf{u}$. The main contribution of the new DG method, as mentioned earlier, is that it correctly and minimally imposes the averages, jumps and penalties, through the correct identification of the physical continuity and discontinuity of the unknown solution $c$ according to the multiple partitions of the interelement boundaries as given by \eqref{eq:correct boundary subset} and through the rigorous justification of the consistency result in Theorem \ref{thm:consistency}. Such DG method and the multiple partitions of the interelement boundaries are not known in the literature, to the authors' knowledge. In addition, the main contribution is that for the exact solution $c$ with the regularity $c\in \prod_{K\in\mathcal{T}_h} W^{1+\hat{\ell},4}(K)$ with a real number $\hat{\ell}$: $0<\hat{\ell}\le\ell$, we derive the error bound $O(h^{\min(\ell,\hat{\ell})})$; when the advection velocity dominates over the diffusion, we obtain the SUPG-type error bound $O(h^{\min(\ell,\hat{\ell})+\frac{1}{2}})$ for the $\ell$th order of the approximation of the finite element space of the DG solution. In particular, when $\hat{\ell}=\ell$, the optimal SUPG-type error bound $O(h^{\ell+\frac{1}{2}})$ follows. As a fundamental technique in the error estimates, a key parameter $\mathcal{D}_K(A,\mathbf{u},\gamma,h_K)$ is defined by \eqref{Eq:Stabilizing parameter-1}. Such parameter allows us to deal with the low regularity 
\eqref{eq:low regularity} 
and the general regularity $c$ of $\prod_{K\in\mathcal{T}_h} W^{1+\hat{\ell},4}(K)$, see Corollary \ref{cor:supg error bound}, Remarks \ref{rem:DK-1} and \ref{rem:DK-2}. Concerning the incompressible miscible flow in porous media like \eqref{eq:diffusion-conveciton-reaction}-\eqref{eq:Darcy}, the realistic and most reasonable regularity for the concentration $c$ and the Darcy velocity $\mathbf{u}$ is $c\in \prod_{m=1}^L W^{1+\hat{\ell},4}(\Omega_m)$ for a partition $\Omega=\cup_{m=1}^L \Omega_m$, $\mathbf{u}\in \prod_{j=1}^J (H^r(D_j))^d$. Under the low regularity advection velocity, the obtained SUPG-type error estimates are new and the first time known in the literature. An interesting feature of the convergence order is that it is independent of the regularity $r$ of the advection velocity. 

The remainder of this paper is organized as follows. In section \ref{sec:Problem statement}, the advection-diffusion-reaction problem is reviewed, with some standard notations. In section \ref{sec:Finite element method}, the finite element method is designed, where the general assumptions on the mesh are formulated, and the crucial multiple partitions of the set of the interelement boundaries are given and the key consistency result is proven. In section \ref{sec:Error estimates}, the error estimates are derived. A conclusion is included.   
    
\section{The advection-diffusion-reaction problem}
\label{sec:Problem statement}

The steady-state, linearized incompressible miscible displacement problem in the concentration can be modelled as the form \eqref{eq:second-order pde}, which is from now on called the advection-diffusion-reaction problem, where the diffusion coefficient amounts to the diffusion-dispersion coefficient, the advection velocity $\mathbf{u}$ is the Darcy velocity and the reaction coefficient $\gamma$ may be the porosity.  

Let $\Omega$ be an open set of $\mathbb{R}^d$ of the $\mathbf{x}=(x_1,x_2,\cdots,x_d)^\top$ coordinates system, $d=2,3$, with Lipschitz continuous boundary $\Gamma$. Introduce the diffusion coefficient $A=(A_{ij})\in\mathbb{R}^{d\times d}$, which is the diffusion-dispersion coefficient in the miscible displacement equation, which is symmetric (i.e., the transpose $A^\top=A$) and satisfies  
\begin{equation}
\xi^\top A(\mathbf{x})\xi\ge 0\quad\forall \xi\in\mathbb{R}^d,\quad\mbox{a.e. $\mathbf{x}\in\bar{\Omega}$}. 
\end{equation}
This vanishing diffusion coefficient would change the nature of the partial differential equations in different subregions of $\Omega$. In addition, the vanishing diffusion coefficient $A$ introduces two portions of the boundary $\Gamma$: the diffusion boundary and the advection boundary. Denote by the diffusion (df) boundary 
\begin{equation}
\Gamma_\text{df}=\{\mathbf{x}\in\Gamma: \mathbf{n}(\mathbf{x})^\top A(\mathbf{x})\mathbf{n}(\mathbf{x})>0\},  
\end{equation}
where ${\bf n}$ denotes a generic outward unit normal vector to any bounded domain. Without loss of generality, we assume that $|\Gamma_\text{df}|>0$ in $\mathbb{R}^{d-1}$ and that $\Gamma_\text{df}=\bar{\Gamma}_{\text{df},D}\cup\bar{\Gamma}_{\text{df},N}$, where $\Gamma_{\text{df},D}\cap\Gamma_{\text{df},N}=\emptyset$,  $|\Gamma_{\text{df},D}|>0$, $|\Gamma_{\text{df},N}|>0$. The subset $\Gamma_{\text{df},D}$ is used for the Dirichlet (D) boundary condition, called the diffusion Dirichlet boundary, while the subset $\Gamma_{\text{df},N}$ is used for the Neumann (N) boundary condition, called the diffusion Neumann boundary. The rest boundary portion  $\Gamma_\text{ad}:=\Gamma\setminus\bar{\Gamma}_\text{df}$ is called the advection (ad) boundary. We assume that $|\Gamma_\text{ad}|>0$. $\Gamma_\text{ad}$ can be defined alternatively by 
\begin{equation}
\Gamma_\text{ad}=\{\mathbf{x}\in\Gamma: \mathbf{n}(\mathbf{x})^\top A(\mathbf{x})\mathbf{n}(\mathbf{x})=0\}. 
\end{equation}
As a consequence (\cite{HoustonSuli2001}), 
\begin{equation}
A(\mathbf{x})\mathbf{n}(\mathbf{x})=0\quad\forall\mathbf{x}\in\Gamma_\text{ad}. 
\end{equation}

Let $W^{s,p}(D)$ be the standard Sobolev space on an open set $D$ of the indices $(s,p)$, $s\in\mathbb{R}, 1\le p\le\infty$. It is equipped with norm $\|\cdot\|_{W^{s,p}(D)}$, while its seminorm is denoted by $|\cdot|_{W^{s,p}(D)}$ for $s\not=0$. For $p=2$ and $s\not=0$, $W^{s,p}(D)=:H^s(D)$, with norm $\|\cdot\|_{W^{s,p}(D)}=:\|\cdot\|_{H^s(D)}$ and seminorm $|\cdot|_{W^{s,p}(D)}=:|\cdot|_{H^s(D)}$; for $p=2,s=0$, $W^{s,p}(D)=:H^0(D)=:L^2(D)$ the standard $L^2$ space, whose inner product and norm are respectively denoted by $(\cdot,\cdot)_{L^2(D)}$ and $\|\cdot\|_{H^0(D)}=|\cdot|_{H^0(D)}=:\|\cdot\|_{L^2(D)}$. Introduce the Hilbert space  $H(\iv;D)=\{\mathbf{v}\in (L^2(D))^d: \iv\mathbf{v}\in L^2(D)\}$, equipped with norm $\|\cdot\|_{H(\iv;D)}=\sqrt{\|\cdot\|_{L^2(D)}^2+\|\iv\cdot\|_{L^2(D)}^2}$, where $\iv\mathbf{v}=\sum_{i=1}^d \partial v_i/\partial x_i$ is the divergence of a vector-valued function $\mathbf{v}=(v_1(\mathbf{x}),v_2(\mathbf{x}),\cdots,v_d(\mathbf{x}))^\top$. 
 
Let $\mathbf{u}$ be the Darcy velocity, which will be always called as the advection velocity from now on. For the advection velocity $\mathbf{u}$, we assume that
\begin{equation}\label{eq:advection Hdiv}
\mathbf{u}\in H(\operatorname{div};\Omega),
\end{equation}
\begin{equation}\label{eq:piecewise regularity}
\mathbf{u}\in \prod_{j=1}^J (H^r(D_j))^d\quad 0<r<1.
\end{equation}
The partition of $\Omega=\cup_{j=1}^J D_j$ only associates with $\mathbf{u}$ in the Darcy flow problem \eqref{eq:Darcy}. Regarding the advection-diffusion-reaction problem under consideration, we  
still assume such partition; in the finite element discretization of the latter problem, because $\mathbf{u}$ is a discontinuous coefficient and we require the meshes to respect this partition.  

The advection velocity introduces an inflow portion and an outflow portion of the advection boundary $\Gamma_\text{ad}$:   
\begin{equation}\label{eq:inflow Gamma-ad-inflow-outflow}
\begin{aligned}
& \Gamma_{\text{ad},-}=\{\mathbf{x} \in \Gamma_\text{ad}: \mathbf{u}(\mathbf{x})\cdot \mathbf{n}(\mathbf{x})< 0\}\quad \text {inflow},
\\
& \Gamma_{\text{ad},+}=\{\mathbf{x} \in \Gamma_\text{ad}: \mathbf{u}(\mathbf{x}) \cdot \mathbf{n}(\mathbf{x}) \ge 0\}\quad \text {outflow}. 
\end{aligned}
\end{equation}
Let
\begin{equation}\label{eq:Dirichlet GammaD-dfD-adinflow}
\Gamma_D:=\Gamma_{\text{df},D}\cup\Gamma_{\text{ad},-}.
\end{equation}
It is used for imposing the Dirichlet boundary condition. To describe the Neumann boundary condition, we further divide $\Gamma_{\text{df},N}$ into two parts: 
\begin{equation}\label{eq:Gamma-dfNinflow-dfNoutflow}
\begin{aligned}
&\Gamma_{\text{df},N,-}=\{\mathbf{x}\in\Gamma_{\text{df},N}: \mathbf{u}(\mathbf{x})\cdot\mathbf{n}(\mathbf{x})<0\},
\\
&\Gamma_{\text{df},N,+}=\{\mathbf{x}\in\Gamma_{\text{df},N}: \mathbf{u}(\mathbf{x})\cdot\mathbf{n}(\mathbf{x})\ge 0\}. 
\end{aligned}
\end{equation} 
With the above boundary portions, we have
\begin{equation}\label{eq:Gamma multiple partition}
\begin{aligned}
\Gamma=&\bar{\Gamma}_\text{df}\cup\bar{\Gamma}_\text{ad}
=\bar{\Gamma}_{\text{df},D}\cup\bar{\Gamma}_{\text{df},N}\cup\bar{\Gamma}_{\text{ad},-}\cup\bar{\Gamma}_{\text{ad},+}
\\
=&\bar{\Gamma}_{\text{df},D}\cup\bar{\Gamma}_{\text{df},N,-}\cup\bar{\Gamma}_{\text{df},N, +}\cup\bar{\Gamma}_{\text{ad},-}\cup\bar{\Gamma}_{\text{ad},+}
\\
=& \bar{\Gamma}_D\cup \bar{\Gamma}_{\text{df},N,-}\cup\bar{\Gamma}_{\text{df},N,+}\cup\bar{\Gamma}_{\text{ad},+}.
\end{aligned}
\end{equation}

Let $g,\gamma$ be given functions and $\kappa,\chi^{\pm}$ be essential Dirichlet and natural Neumann boundary data, respectively. Consider the advection-diffusion-reaction problem as follows: Find $c$ such that

\begin{equation}\label{Eq:BVP}
\iv(-A\nabla c+{\bf u}c)+\gamma c=g\quad\mbox{in $\Omega$},
\end{equation}

\begin{equation}\label{Eq:Boundary essential and natrual condition}
c|_{\Gamma_D}
=\kappa, \quad (-A\nabla c+\mathbf{u}c)\cdot\mathbf{n}|_{\Gamma_{\text{df},N,-}}=\chi^{-}, \quad -A\nabla c\cdot{\bf n}|_{\Gamma_{\text{df},N,+}}=\chi^{+}.
\end{equation}

From the fac that $\iv({\bf u}c)=c \iv{\bf u} +{\bf u}\cdot\nabla c=c \iv(\frac{\bf u}{2})+\frac{\bf u}{2}\cdot\nabla c+\frac{\iv({\bf u}c)}{2}$, we, instead of \eqref{Eq:BVP}, use the equation
\begin{equation}\label{Eq:Equivalent PDE}
-\iv(A\nabla c)+\frac{\bf u}{2}\cdot\nabla c+\frac{\iv({\bf u}c)}{2}+(\gamma+\iv(\frac{\bf u}{2}))c=g.
\end{equation}
Both are equivalent. We assume that the reaction coefficient $\gamma_0:=
\gamma+\frac{\iv\mathbf{u}}{2}\ge 0$ on $\Omega$ ($\gamma_0$ may vanish somewhere in $\Omega$).  Throughout this paper, we have assumed that the total physical flux 
\begin{equation}\label{eq:normal continuity}
-A\nabla c+\mathbf{u}c\in H(\iv;\Omega).
\end{equation}

\section{Finite element method}
\label{sec:Finite element method}

\subsection{Mesh}

Let ${\cal T}_h=\{K\}$ denote a mesh of $\Omega$, composed of elements (denoted by a generic notation $K$) of Lipschitz-continuous boundary $\partial K$, so that $\bar{\Omega}=\cup_{K\in\mathcal{T}_h}\bar{K}$, where $|K|>0$ in $\mathbb{R}^d$ and $|\partial K|>0$ in $\mathbb{R}^{d-1}$, and $h=\max_{K\in\mathcal{T}_h} h_K$, and $h_K$ is the diameter of the element $K\in{\cal T}_h$. Let $\mathcal{F}_h^0$ denote the set of all the interior intersections with positive measurements in $\mathbb{R}^{d-1}$ of two adjacent elements, i.e., for $\Lambda\in\mathcal{F}_h^0$, there are two elements such that $\Lambda=\partial K_1\cap\partial K_2$ and $|\Lambda|>0$ in $\mathbb{R}^{d-1}$ and $|\Lambda\cap\Gamma|=0$. Let $\Sigma_0=\{\mathbf{x}\in\bar{\Omega}: \mathbf{x}\in\Lambda \text{ for some } \Lambda\subset\mathcal{F}_h^0\}$. Let $\mathcal{F}_{\Gamma,h}$ denote the set of all the intersections with positive measurements in $\mathbb{R}^{d-1}$ of the element boundary $\partial K$ of every element $K$ of $\mathcal{T}_h$ with $\Gamma$. Put $\mathcal{F}_h=\mathcal{F}_h^0\cup\mathcal{F}_{\Gamma,h}$. Denote by $h_\Lambda$ the diameter of $\Lambda\in\mathcal{F}_h$. Trivially, $h_\Lambda\le h_K$, $|\Lambda|\le C h_\Lambda^{d-1}$, and $|K|\le C h_K^d$ for some constant $C$. Throughout this paper, $C$  denotes a generic constant. It may take different values at different occurrences,but it does not depend on $h,K,A,\mathbf{u},\gamma,c$ and the data $g,\kappa,\chi^\pm$.  

We make some assumptions under which the mesh is allowed to have essentially arbitrarily shaped elements and to be arbitrarily nonconforming.  

\textit{Assumption 1} (Shape-regularity) For every $K\in\mathcal{T}_h$, 
\begin{equation}\label{eq:shape-regularity}
\begin{aligned}
& |\partial K| h_K \le C |K|.
\end{aligned}
\end{equation}

\textit{Assumption 2} (Condition on the nonconformity of mesh) 
\begin{equation}\label{eq:nonconformity}
\begin{aligned}
& |\Lambda|^\frac{d-2}{d-1} h_{K}^2\le C |K|\quad\forall\Lambda\in\mathcal{F}_h, \Lambda\subseteq\partial K. 
\end{aligned}
\end{equation}  

\textit{Assumption 3} (Local trace inequality)  
\begin{equation}\label{eq:local trace inequality}
\begin{aligned}
\frac{|K|}{|S|} \|q\|_{L^2(S)}^2\le C (\|q\|_{L^2(K)}^2+h_K^{2t} 
|q|_{H^t(K)}^2)
\\
\forall S\subseteq\partial K,|S|>0,\quad \frac{1}{2}<t\le 1,\quad q\in H^t(K). 
\end{aligned}
\end{equation}
\begin{equation}\label{eq:local trace inequality-finitedimensionalspace}
\frac{|K|}{|S|} \|q\|_{L^2(S)}^2\le C \|q\|_{L^2(K)}^2
\quad \mbox{$q$ in a finite-dimensional space over $K\in\mathcal{T}_h$}. 
\end{equation}
\begin{equation}\label{eq:local trace ienquality wsp}
\begin{aligned}
&\frac{|K|}{|S|} \|q\|_{L^2(S)}^2\le C (\|q\|_{L^2(K)}^2+h_K^{2t} |K|^\frac{1}{2}|q|_{W^{t,4}(K)}^2)
\\
&\forall S\subseteq\partial K,|S|>0,\quad \frac{1}{4}<t\le 1,\quad q\in W^{t,4}(K). 
\end{aligned}
\end{equation}

\begin{equation}\label{eq:local trace L1}
\begin{aligned}
\frac{|K|^\frac{1}{2}}{|S|} \|q\|_{L^1(S)}\le C (\|q\|_{L^2(K)}+h_K^t |q|_{H^t(K)})
\\
\forall S\subseteq\partial K,|S|>0,\quad 0<t\le 1,\quad q\in H^t(K). 
\end{aligned}
\end{equation}

\textit{Assumption 4} (Local inverse inequality) For any function $q$ in a finite-dimensional space over $K\in\mathcal{T}_h$,  
\begin{equation}\label{eq:local inverse H1L2}
|q|_{H^1(K)}\le C 
h_K^{-1}  
\|q\|_{L^2(K)},
\end{equation}

\textit{Assumption 5} (Local interpolation property) For any $K\in\mathcal{T}_h$, let $P_\ell(K)$, $\ell\ge 0$, denote the space of polynomials over $K$ of total degree $\ell$, and let $\pi_K v\in P_\ell(K)$ denote an interpolation of $v$. Let $S\subset\partial K$ be any portion with $|S|>0$ in $\mathbb{R}^{d-1}$. We assume that $\pi_Kv$ satisfies the following interpolation error estimates: 

\begin{equation}\label{Eq:interpolation error-0}
\begin{aligned}
& h_K^t |v-\pi_K v|_{H^t(K)}\le C h_K^s \|v\|_{H^s(K)},  
\quad
0\le t\le s,\quad 0\le s\le 1+\ell, \quad\text{ $v\in H^s(K)$},
\\
&\left(\frac{|K|}{|S|}\right)^\frac{1}{2}\|v-\pi_Kv\|_{L^2(S)}\le C h_K^s \|v\|_{H^s(K)}
,\quad 
\frac{1}{2}< s\le 1+\ell, \quad\text{ $v\in H^s(K)$};
\end{aligned}
\end{equation}

\begin{equation}\label{Eq:interpolation error-1}
\begin{aligned}
\|v-\pi_Kv\|_{W^{t,q}(K)}\le C |K|^{1/q-1/p} h_K^{s-t} \|v\|_{W^{s,p}(K)}, \quad
0\le s\le 1+\ell, \quad\text{ $v\in W^{s,p}(K)$},
\end{aligned}
\end{equation}
where $p,q\ge 1$ and $s\ge t\ge 0$ are real numbers such that $W^{s,p} \hookrightarrow W^{t,q}$, 
$
1\le q\le q^*,\quad 0\le t\le t^*,\quad 1/q*=1/p-(s-t^*)/d>0$. In this assumption, if necessary, we may replace $\pi_K$, $\|v\|_{H^s(K)}$ and $\|v\|_{W^{s,p}(K)}$ by $\pi_{\Delta_K}$, $\|v\|_{H^s(\Delta_K)}$ and $\|v\|_{W^{s,p}(\Delta_K)}$, respectively, where $\Delta_K$ stands for some patch of $K$.  

\begin{remark}
The above assumptions are used for the purpose of analysis. They implicitly impose conditions on the mesh.  If the mesh is of simplexes (triangles in two dimensions or tetrahedra in three dimensions), the above assumptions can be easily verified under the classical shape-regularity condition $h_K/\rho_K\le C$ where $\rho_K$ denotes the diameter of the largest ball inscribed $K$ and under the classical condition $h_K\le C h_\Lambda$ on the nonconformity of the mesh.   
\end{remark}

\subsection{Finite element method}

We define the finite element space for approximating the solution 
\begin{equation}\label{Eq:FEM space of solution}
U_h=\{v\in L^2(\Omega): v|_K\in P_\ell(K),\forall K\in{\cal T}_h\},\quad \ell\ge 1. 
\end{equation}

For $\Lambda\in\mathcal{F}_h^0$, $\Lambda=\partial K_1\cap \partial K_2$, we assign ${\bf n}_\Lambda$ to denote the unit normal vector to $\Lambda$, with a fixed direction orienting from $K_1$ to $K_2$. For $\Lambda\in{\cal F}_{\Gamma,h}$, then ${\bf n}_\Lambda$ still denotes the outward unit vector to $\Lambda$. Define averages and jumps as follows:

\begin{equation}\label{Eq:Average and Jump of scalar variable}
\Lbr v\Rbr =\dfrac{v|_{K_1}+v|_{K_2}}{2}, \quad \Lbrack v\Rbrack =v|_{K_1}-v|_{K_2},
\end{equation}

\begin{equation}\label{Eq:Average and Jump of normal trace}
\Lbr {\bm\tau}\cdot{\bf n}_\Lambda\Rbr =\dfrac{({\bm\tau}|_{K_1}+{\bm\tau}|_{K_2})\cdot{\bf n}_\Lambda}{2}, \quad \Lbrack {\bm\tau}\cdot{\bf n}_\Lambda\Rbrack =({\bm\tau}|_{K_1}-{\bm\tau}|_{K_2})\cdot{\bf n}_\Lambda.
\end{equation}
If $\Lambda\in{\cal F}_{\Gamma,h}$ and $\Lambda\subset\partial K$, then 
$
\Lbr v\Rbr =
\Lbrack v\Rbrack =v|_K$,  
and $
\Lbr {\bm\tau}\cdot{\bf n}_\Lambda\Rbr =
\Lbrack {\bm\tau}\cdot{\bf n}_\Lambda\Rbrack ={\bm\tau}\cdot{\bf n}
$, where ${\bf n}={\bf n}_\Lambda$ is the outward unit normal vector to $\Gamma$ (and to $\partial K$). In addition, define the local inflow boundary and the local outflow boundary of $\partial K$: 

\begin{equation*}
\begin{aligned}
&\partial_- K=\{{\bf x}\in \partial K:{\bf u}({\bf x})\cdot{\bf n}({\bf x}) < 0\},
\\
&\partial_+ K=\{{\bf x}\in \partial K:{\bf u}({\bf x})\cdot{\bf n}({\bf x})\ge 0\}.
\end{aligned}
\end{equation*}
Obviously, if $\Lambda=\partial K\cap \partial K_1$ and $\Lambda\subseteq\partial_- K$, then $\Lambda\subseteq \partial_+ K_1$. Introduce the notations of the inner products:
$
(u,v)_{L^2({\cal T}_h)}=\sum_{K\in{\cal T}_h}(u,v)_{L^2(K)}$, $
\langle p,q\rangle_{L^2({\cal G}_h)}=\sum_{\Lambda\in{\cal G}_h} \langle p,q\rangle_{L^2(\Lambda)}$, 
where ${\cal G}_h$ represents any subset of ${\cal F}_h$, and the induced norms:$
\|v\|_{L^2({\cal T}_h)}^2=\sum_{K\in{\cal T}_h}\|v\|_{L^2(K)}^2$, $
\|q\|_{L^2({\cal G}_h)}^2=\sum_{\Lambda\in{\cal G}_h} \|q\|_{L^2(\Lambda)}^2$. 

We are now able to define the finite element method. Since the diffusion may vanish somewhere of $\Omega$ and since the discretization of the advection term already controls the jumps in the normal direction of the advection velocity $\mathbf{u}$ across every interelement boundary, we only need to artificially penalize the jumps across those element boundaries with the nonzero diffusion. Thus, we first divide  $\mathcal{F}_h^0$ into two subsets:   
\begin{equation}\label{eq:diffusion element boundary-0}
\mathcal{F}_h^0=\mathcal{F}_h^{\text{df-df},0}\cup (\mathcal {F}_h^0\setminus \mathcal{F}_h^{\text{df-df},0}), 
\end{equation} 
where $\mathcal{F}_h^{\text{df-df},0}$ denotes the subset of the diffusion-diffusion (df-df) element boundaries of $\mathcal{F}_h^0$:   
\begin{equation}\label{eq:diffusion element boundary}
\begin{aligned}
\mathcal{F}_h^{\text{df-df},0}=\{& \Lambda\in\mathcal{F}_h^0: \Lambda=\partial K_1\cap\partial K_2, 
\\
& \text{both }\mathbf{n}_{\partial K_1}(\mathbf{x})^\top A(\mathbf{x})\mathbf{n}_{\partial K_1}(\mathbf{x})>0, \text{ a.e. } \mathbf{x}\in\bar{\Lambda}
\\
&\text{ and }\mathbf{n}_{\partial K_2}(\mathbf{x})^\top A(\mathbf{x})\mathbf{n}_{\partial K_2}(\mathbf{x}))>0, \text{ a.e. }\mathbf{x}\in\bar{\Lambda}\},
\end{aligned}
\end{equation}   
where $\mathbf{n}_{\partial K_i}$ denotes the outward unit normal to $\partial K_i$, $i=1,2$, while $\mathcal{F}_h^{\text{ad},0}:=\mathcal {F}_h^0\setminus \mathcal{F}_h^{\text{df-df},0}$ denotes the subset of advection (ad) element boundaries
\begin{equation}\label{eq:advection element boundary}
\begin{aligned}
\mathcal{F}_h^{\text{ad},0}=\{& \Lambda\in\mathcal{F}_h^0: \Lambda=\partial K_1\cap\partial K_2, 
\\
& \text{either } \mathbf{n}_{\partial K_1}(\mathbf{x})^\top A(\mathbf{x})\mathbf{n}_{\partial K_1}(\mathbf{x})>0 \text{ and }\mathbf{n}_{\partial K_2}(\mathbf{x})^\top A(\mathbf{x})\mathbf{n}_{\partial K_2}(\mathbf{x}))=0 \text{ a.e. } \mathbf{x}\in \bar{\Lambda},
\\
& \text{ or }
\mathbf{n}_{\partial K_1}(\mathbf{x})^\top A(\mathbf{x})\mathbf{n}_{\partial K_1}(\mathbf{x}))=0 \text{ and } \mathbf{n}_{\partial K_2}(\mathbf{x})^\top A(\mathbf{x})\mathbf{n}_{\partial K_2}(\mathbf{x})>0 \text{ a.e. } \mathbf{x}\in\bar{\Lambda},
\\
&\text{ or } \mathbf{n}_{\partial K_1}(\mathbf{x})^\top A(\mathbf{x})\mathbf{n}_{\partial K_1}(\mathbf{x})=0 \text{ and }\mathbf{n}_{\partial K_2}(\mathbf{x})^\top A(\mathbf{x})\mathbf{n}_{\partial K_2}(\mathbf{x})=0 \text{ a.e. } \mathbf{x}\in\bar{\Lambda}\}.
\end{aligned}
\end{equation} 
We only need to identify $\mathcal{F}_h^{\text{df-df},0}$, \textit{while other subsets are used only for theoretical analysis.} In addition, we need to identify the subset $\mathcal{F}_{\text{df},D,h}$ of the element boundaries on $\Gamma_{\text{df},D}$, which is the set of the intersections of the element boundary $\partial K$ of every $K\in\mathcal{T}_h$ with $\Gamma_{\text{df},D}$.   
  
Let $A_{\max,K}$ denote the interior local maximum over $\bar{K}$ of $K\in\mathcal{T}_h$, a constant defined by 
\begin{equation}
\xi^\top A(\mathbf{x}) \xi\le A_{\max,K} |\xi|^2\quad\forall\xi\in\mathbb{R}^d\quad\mbox{a.e. $\mathbf{x}\in K$}.
\end{equation}
Restricted to both sides of $\Lambda=\partial K_1\cap\partial K_2\in \mathcal{F}_h^{\text{df-df},0}$ or to $\Lambda=\partial K\in\mathcal{F}_{\text{df},D,h}$,  
the above constants can be similarly defined, respectively denoted by $A_{\max,\Lambda, \partial K_1}$ and $A_{\max,\Lambda, \partial K_2}$ or $A_{\max,\Lambda,\partial K}$. We must have 
\begin{equation*}
\begin{aligned}
&\min(A_{\max,\Lambda, \partial K_1}, A_{\max,\Lambda, \partial K_2})>0\quad\text{ on } \Lambda=\partial K_1\cap\partial K_2\in\mathcal{F}_h^{\text{df-df},0},
\\
&A_{\max,\Lambda,\partial K}>0\quad\text{ on }\Lambda 
\in\mathcal{F}_{\text{df},D,h}. 
\end{aligned}
\end{equation*}
Let
\begin{equation}
\begin{aligned}
& A_{\Lambda}:=\frac{A_{\max,\Lambda,\partial K_1}+A_{\max,\Lambda,\partial K_2}}{2} 
\quad \text{ where } \Lambda=\partial K_1\cap\partial K_2\in\mathcal{F}_h^{\text{df-df},0},
\\
& A_{\Lambda}:=A_{\max,\Lambda,\partial K}\quad \text{ where } \Lambda=\partial K\cap\Gamma_{\text{df},D}\in\mathcal{F}_{\text{df},D,h}.
\end{aligned}
\end{equation}
Introduce the inflow subset $\Gamma_{\text{df},D,-}$ of $\Gamma_{\text{df},D}$, defined by
\begin{equation}\label{eq:inflow df-D}
\Gamma_{\text{df},D,-}=\{\mathbf{x}\in\Gamma_{\text{df},D}: \mathbf{u}(\mathbf{x})\cdot\mathbf{n}(\mathbf{x})<0\}. 
\end{equation}

For a varying, low regularity advection velocity $\mathbf{u}$, we need to define a parameter associated with the advection term. Such parameter is called the stabilizing parameter in the literature of stabilized finite element method. We then need to define a stabilization which is similar to the one in the SUPG method. We can thus obtain the SUPG-type stability and the SUPG-type error bound.       
Let $K\in\mathcal{T}_h$ and introduce
\begin{equation}\label{eq:epsilon diffusion}
\begin{aligned}
&\Upsilon_{\text{df},\partial K}=\{\Lambda\in\mathcal{F}_h^{\text{df-df},0}\cup\mathcal{F}_{\text{df},D,h}: \Lambda\subset\partial K\}. 
\end{aligned}
\end{equation}
Then, define 
\begin{equation}
\varepsilon_{\text{df},K}=\max\{A_{\max,K}, A_{
\Lambda 
} \text{ for all }\Lambda\in\Upsilon_{\text{df},\partial K}\}. 
\end{equation}
Let $\alpha>0$ be a constant, which is determined later mainly by the constant in the inverse inequality.  Define a SUPG-type bilinear form  
\begin{equation}\label{eq:stabilization}
S_h(c,v)=\sum_{K\in\mathcal{T}_h} (-\iv(A\nabla c)+\frac{\bf u}{2}\cdot\nabla c+\frac{\iv({\bf u}c)}{2}+\gamma_0 c, \tau_K (\frac{\bf u}{2}\cdot\nabla v+\frac{\iv({\bf u}v)}{2}-\gamma_0 v))_{L^2(K)}, 
\end{equation}
where, inspired by \cite{FrancaValentine2000}, $\tau_K$  
is defined by 
\begin{equation}\label{eq:stabilizing parameter pointwise}
\tau_K(\mathbf{x})=\frac{\alpha h_K^2 }{\varepsilon_{\text{df},K} 
+h_K |\mathbf{u}(\mathbf{x})|_2+h_K^2|\iv\mathbf{u}(\mathbf{x})|+h_K^2 |\gamma_0(\mathbf{x})|}\quad\mathbf{x}\in\bar{\Omega},
\end{equation}
where 
$|\cdot|_2$ and $|\cdot|$ stand for the Euclidean  
$2$-norm and absolute value, respectively. The consistent SUPG-type linear form is
\begin{equation}\label{eq:stabilization linear form}
T_h(v)=\sum_{K\in\mathcal{T}_h} (g, \tau_K (\frac{\bf u}{2}\cdot\nabla v+\frac{\iv({\bf u}v)}{2}-\gamma_0 v))_{L^2(K)}.
\end{equation} 
 
Define the bilinear form and the linear form respectively as follows:
\begin{equation}\label{Eq:Bilinear form minus}
\begin{aligned}
\mathscr{L}_h(c,v) :=&
 \sum_{K\in\mathcal{T}_h} (A\nabla c,\nabla v)_{L^2(K)}-\sum_{\Lambda\in\mathcal{F}_h^{\text{df-df},0}\cup \mathcal{F}_{\text{df},D,h}} \langle \Lbr (A\nabla c)\cdot\mathbf{n}_\Lambda\Rbr, \Lbrack v\Rbrack\rangle_{L^2(\Lambda)} 
\\
&+\sum_{\Lambda\in\mathcal{F}_h^{\text{df-df},0}\cup \mathcal{F}_{\text{df},D,h}} \langle \Lbr (A\nabla v)\cdot\mathbf{n}_\Lambda\Rbr, \Lbrack c\Rbrack\rangle_{L^2(\Lambda)}
\\
&+ \sum_{\Lambda\in\mathcal{F}_h^{\text{df-df},0}} |\Lambda|^{-\frac{1}{d-1}} 
A_{\Lambda} \langle\Lbrack c\Rbrack, \Lbrack v\Rbrack\rangle_{L^2(\Lambda)}
+ \sum_{\Lambda\in\mathcal{F}_{\text{df},D,h}} |\Lambda|^{-\frac{1}{d-1}} 
A_{\Lambda} \langle c, v\rangle_{L^2(\Lambda)}
\\
&+\sum\limits_{K\in{\cal T}_h}(\frac{\bf u}{2}\cdot\nabla c+\frac{\iv({\bf u}c)}{2},v)_{L^2(K)}-\sum\limits_{K\in{\cal T}_h} \langle v, {\bf u}\cdot{\bf n} \Lbrack c\Rbrack \rangle_{L^2(\partial_- K)}
\\
&+(
\gamma_0 c,v)_{L^2({\cal T}_h)},
\end{aligned}
\end{equation}
\begin{equation}\label{Eq:Bilinear form}
\mathscr{L}_h^S(c,v):=\mathscr{L}_h(c,v)+S_h(c,v),
\end{equation}

\begin{equation}\label{Eq:Linear form minus}
\begin{aligned}
\mathscr{R}_h(v):=&(g,v)_{L^2(\Omega)}+\sum_{\Lambda\in\mathcal{F}_{\text{df},D,h}} \langle (A\nabla v)\cdot\mathbf{n},\kappa\rangle_{L^2(\Lambda)}
\\
&+ \sum_{\Lambda\in\mathcal{F}_{\text{df},D,h}} |\Lambda|^{-\frac{1}{d-1}} 
A_{\Lambda} 
\langle\kappa, v\rangle_{L^2(\Lambda)}  
-\langle v,\mathbf{u}\cdot\mathbf{n} \kappa\rangle_{L^2(\Gamma_{\text{ad},-}\cup\Gamma_{\text{df},D,-})} 
\\
& 
-\langle v,\chi^{-}\rangle_{L^2(\Gamma_{\text{df},N,-})}
-\langle v,\chi^{+}\rangle_{L^2(\Gamma_{\text{df},N,+})}, 
\end{aligned}
\end{equation}
\begin{equation}\label{Eq:Linear form}
\mathscr{R}_h^T(v):=\mathscr{R}_h(v)+T_h(v).
\end{equation}

The DG problem reads as follows: Find $c_h\in U_h$ such that

\begin{equation}\label{Eq:FEM-2 problem}
\begin{array}{l}
\mathscr{L}_h^S(c_h,v_h)=\mathscr{R}_h^T(v_h)\quad\forall v_h\in U_h.
\end{array}
\end{equation}

We can further reduce the amount of penalties. Introduce a subset of $\mathcal{F}_h^{\text{\rm df-df},0}$, called the subset of the diffusion-diffusion with diffusion-dominated (dfd)  element boundaries, defined by
\begin{equation}\label{eq:df-df-dfd set}
\begin{aligned}
&\mathcal{F}_h^{\text{\rm df-df-dfd},0}=\{\Lambda\in\mathcal{F}_h^{\text{\rm df-df},0}: |\Lambda|^\frac{1}{d-1} 
|\mathbf{u}(\mathbf{x})\cdot\mathbf{n}(\mathbf{x})|< A_\Lambda 
,\text{ a.e. } \mathbf{x}\in\bar{\Lambda}\}.
\end{aligned}
\end{equation}  
Similarly, for the diffusion Dirichlet element boundary set $\mathcal{F}_{\text{\rm df,D},h}$, we introduce the subset of the diffusion with diffusion-dominated Dirichlet element boundaries, defined by   
\begin{equation}\label{eq:df-dfd-D set}
\begin{aligned}
&\mathcal{F}_{\text{\rm df-dfd,D},h}=\{\Lambda\in\mathcal{F}_{\text{\rm df},D,h}: |\Lambda|^\frac{1}{d-1}  
|\mathbf{u}(\mathbf{x})\cdot\mathbf{n}(\mathbf{x})|< A_\Lambda
,\text{ a.e. } \mathbf{x}\in\bar{\Lambda}\}.
\end{aligned}
\end{equation}
Then, the penalization term is replaced by 
\begin{equation}\label{eq:reduced penalty}
\begin{aligned}
&\sum_{\Lambda\in\mathcal{F}_h^{\text{\rm df-df-dfd},0}} |\Lambda|^{-\frac{1}{d-1}} 
A_\Lambda\langle\Lbrack c\Rbrack, \Lbrack v\Rbrack\rangle_{L^2(\Lambda)}+\sum_{\Lambda\in\mathcal{F}_{\text{\rm df-dfd,D},h}} |\Lambda|^{-\frac{1}{d-1}} 
A_\Lambda\langle c, v \rangle_{L^2(\Lambda)}.
\end{aligned}
\end{equation} 
Denote by $\mathcal{F}_h^{\text{\rm df-df-add},0}:=\mathcal{F}_h^{\text{\rm df-df},0}\setminus \mathcal{F}_h^{\text{\rm df-df-dfd},0}$. The subset $\mathcal{F}_h^{\text{\rm df-df-add},0}$ is called the subset of the diffusion-diffusion with advection-dominated (add) element boundaries. Denote by $\mathcal{F}_{\text{\rm df-add,D},h}:=\mathcal{F}_{\text{\rm df,D},h} \setminus \mathcal{F}_{\text{\rm df-dfd,D},h}$, called the subset of the diffusion with advection-dominated Dirichlet element boundaries. These two subsets read as follows: 
\begin{equation}\label{eq:advection-dominated df-df-0}
\begin{aligned}
&\mathcal{F}_h^{\text{\rm df-df-add},0}=\{\Lambda\in\mathcal{F}_h^{\text{\rm df-df},0}: |\Lambda|^\frac{1}{d-1} 
|\mathbf{u}(\mathbf{x})\cdot\mathbf{n}(\mathbf{x})\ge A_\Lambda 
,\text{ a.e. } \mathbf{x}\in\bar{\Lambda}\}.
\end{aligned}
\end{equation}  
\begin{equation}\label{eq:advection-dominated df-df-1}
\begin{aligned}
&\mathcal{F}_{\text{\rm df-add,D},h}=\{\Lambda\in\mathcal{F}_{\text{\rm df,D},h}: |\Lambda|^\frac{1}{d-1} 
|\mathbf{u}(\mathbf{x})\cdot\mathbf{n}(\mathbf{x})|\ge A_\Lambda
,\text{ a.e. } \mathbf{x}\in\bar{\Lambda}\}.
\end{aligned}
\end{equation}
On the diffusion-diffusion with advection-dominated element boundaries of $\mathcal{F}_h^{\text{\rm df-df-add},0}$ and the diffusion with advection-dominated diffusion Dirichlet element boundaries of $\mathcal{F}_{\text{\rm df-add,D},h}$, we do not need to penalize anything. In other words, we need only penalize the diffusion dominated element boundaries in the set of diffusion-diffusion element boundaries and in the set of the diffusion Dirichlet element boundaries. This is consistent with the DG method of a pure diffusion problem (i.e., $\mathbf{u}=0$).  For convenience, we introduce 
\begin{equation}\label{Eq:Bilinear form-1}
\begin{aligned}
\mathscr{L}_{{\rm dfd},h}^S(c,v):=&
 \sum_{K\in\mathcal{T}_h} (A\nabla c,\nabla v)_{L^2(K)}-\sum_{\Lambda\in\mathcal{F}_h^{\text{\rm df-df},0}\cup \mathcal{F}_{\text{\rm df,D},h}} \langle \Lbr (A\nabla c)\cdot\mathbf{n}_\Lambda\Rbr, \Lbrack v\Rbrack\rangle_{L^2(\Lambda)} 
\\
&+\sum_{\Lambda\in\mathcal{F}_h^{\text{\rm df-df},0}\cup \mathcal{F}_{\text{\rm df,D},h}} \langle \Lbr (A\nabla v)\cdot\mathbf{n}_\Lambda\Rbr, \Lbrack c\Rbrack\rangle_{L^2(\Lambda)}
\\
&+ \sum_{\Lambda\in\mathcal{F}_h^{\text{\rm df-df-dfd},0}} |\Lambda|^{-\frac{1}{d-1}} 
A_{\Lambda} \langle\Lbrack c\Rbrack, \Lbrack v\Rbrack\rangle_{L^2(\Lambda)}
+ \sum_{\Lambda\in\mathcal{F}_{\text{\rm df-dfd},D,h}} |\Lambda|^{-\frac{1}{d-1}} 
A_{\Lambda} \langle c, v\rangle_{L^2(\Lambda)}
\\
&+\sum\limits_{K\in{\cal T}_h}(\frac{\bf u}{2}\cdot\nabla c+\frac{\iv({\bf u}c)}{2},v)_{L^2(K)}-\sum\limits_{K\in{\cal T}_h} \langle v, {\bf u}\cdot{\bf n} \Lbrack c\Rbrack \rangle_{L^2(\partial_- K)}
\\
&+(\gamma_0 c,v)_{L^2({\cal T}_h)}+S_h(c,v),
\end{aligned}
\end{equation}
while the linear form in the right-hand side is the same as  the one in \eqref{Eq:Linear form}.  
The DG problem is to find $c_h\in U_h$ such that 
\begin{equation}\label{Eq:FEM-2 problem-1}
\mathscr{L}_{{\rm dfd},h}^S(c_h,v_h)=\mathscr{R}_h^T(v_h)\quad\forall v_h\in U_h.
\end{equation}
Note that \eqref{Eq:FEM-2 problem-1} differs from \eqref{Eq:FEM-2 problem} by the penalty terms in the bilinear forms. 

\begin{remark}
As far as we know, the 
DG methods in the few literature (\cite{HoustonSchwabSuli2002}, \cite{ProftRiviere2009}, \cite{DiPietroDroniouErn2015}, \cite{DiPietroErnGuermond2008}) use the `full' 
averages, jumps, and penalizations in the sense that all 
interelement boundaries of $\Omega$ are involved with averages, jumps and penalties.    
\end{remark}

\begin{remark}\label{rem:multiple partition}
Although we have developed  multiple partitions of the set of the interelement boundaries, including the further multiple partitions in the next subsection, the multiple partitions are only used for the theoretical analysis, except that we need only to identify the set $\mathcal{F}_h^{\text{df-df},0}$ of the diffusion-diffusion interelement boundaries. Such set can be easily determined in practice. For example, let $A=\varepsilon(\mathbf{x})I$, here $I$ is the identity matrix and $\varepsilon(\mathbf{x})\ge 0$ a.e. $\mathbf{x}\in\bar{\Omega}$; to find the set $\mathcal{F}_h^{\text{df-df},0}$ one needs only to find where $\varepsilon>0$; in practice, the regions where $\varepsilon>0$ can be known a priori.     
\end{remark}

\subsection{Consistency and stability}
\label{subsec:Stability and consistency}

Now, we develop the analysis for the consistency and the stability for the DG problem \eqref{Eq:FEM-2 problem-1}. The analysis is straightforwardly applicable to the DG problem \eqref{Eq:FEM-2 problem}, because they only differ in the artificial penalty terms.   

We first examine the consistency property. The following consistency result rigorously justifies the correct and minimal averages, jumps and penalties in the proposed DG method. For that purpose, we need to further 
divide the subset $\mathcal{F}_h^{\text{ad},0}$ as given by \eqref{eq:advection element boundary}. Let $\mathcal{F}_h^{\text{ad},0}=\mathcal{F}_h^{\text{ad-df},0}\cup\mathcal{F}_h^{\text{ad-ad},0}$:
\begin{equation}\label{eq:ad-df set}
\begin{aligned}
\mathcal{F}_h^{\text{ad-df},0}=\{& \Lambda\in\mathcal{F}_h^{\text{ad},0}: \Lambda=\partial K_1\cap\partial K_2, 
\\
& \text{either } \mathbf{n}_{\partial K_1}(\mathbf{x})^\top A(\mathbf{x})\mathbf{n}_{\partial K_1}(\mathbf{x})>0 \text{ and }\mathbf{n}_{\partial K_2}(\mathbf{x})^\top A(\mathbf{x})\mathbf{n}_{\partial K_2}(\mathbf{x}))=0 \text{ a.e. } \mathbf{x}\in \bar{\Lambda},
\\
& \text{ or }
\mathbf{n}_{\partial K_1}(\mathbf{x})^\top A(\mathbf{x})\mathbf{n}_{\partial K_1}(\mathbf{x}))=0 \text{ and } \mathbf{n}_{\partial K_2}(\mathbf{x})^\top A(\mathbf{x})\mathbf{n}_{\partial K_2}(\mathbf{x})>0 \text{ a.e. } \mathbf{x}\in\bar{\Lambda}\}.
\end{aligned}
\end{equation}
\begin{equation}\label{eq:ad-ad set}
\begin{aligned}
\mathcal{F}_h^{\text{ad-ad},0}=\{& \Lambda\in\mathcal{F}_h^{\text{ad},0}: \Lambda=\partial K_1\cap\partial K_2, 
\\
& \text{ both }\mathbf{n}_{\partial K_1}(\mathbf{x})^\top A(\mathbf{x})\mathbf{n}_{\partial K_1}(\mathbf{x}))=0 \text{ and } \mathbf{n}_{\partial K_2}(\mathbf{x})^\top A(\mathbf{x})\mathbf{n}_{\partial K_2}(\mathbf{x})=0 \text{ a.e. } \mathbf{x}\in \bar{\Lambda}\}.
\end{aligned}
\end{equation}
Moreover, $\mathcal{F}_h^{\text{ad-df},0}$ is divided into inflow and outflow two parts: $\mathcal{F}_h^{\text{ad-df},0}=\mathcal{F}_h^{\text{ad-df},0,-}\cup\mathcal{F}_h^{\text{ad-df},0,+}$
\begin{equation}\label{eq:ad-df-inflow-outflow}
\begin{aligned}
& \mathcal{F}_h^{\text{ad-df},0,-}=\{\Lambda\in\mathcal{F}_h^{\text{ad},0}: \mathbf{u}(\mathbf{x})\cdot\mathbf{n}(\mathbf{x})<0,\text{ a.e. } \mathbf{x}\in\bar{\Lambda}\},
\\
& \mathcal{F}_h^{\text{ad-df},0,+}=\{\Lambda\in\mathcal{F}_h^{\text{ad},0}: \mathbf{u}(\mathbf{x})\cdot\mathbf{n}(\mathbf{x})\ge 0,\text{ a.e. } \mathbf{x}\in\bar{\Lambda}\}.
\end{aligned}
\end{equation}
\textit{We must emphasize that these subsets are introduced only for the theoretical purpose; they do not enter into the DG method and the implementation.} Then, 
\begin{equation}\label{eq:correct boundary subset}
\mathcal{F}_h^0=\mathcal{F}_h^{\text{df-df},0}\cup\mathcal{F}_h^{\text{ad-df},0,+}\cup\mathcal{F}_h^{\text{ad-df},0,-}\cup\mathcal{F}_h^{\text{ad-ad},0}.
\end{equation} 

\begin{theorem}\label{thm:consistency}
Let $c$ be the exact solution of \eqref{Eq:BVP} or \eqref{Eq:Equivalent PDE}-\eqref{Eq:Boundary essential and natrual condition}. 
In addition to the assumptions \eqref{eq:advection Hdiv}, \eqref{eq:normal continuity}, assume that  
\begin{subequations}
\begin{equation}\label{eq:H1 continuity}
\Lbrack c\Rbrack|_{\Lambda}=0\quad\forall\Lambda\in \mathcal{F}_h^{\text{\rm df-df},0},
\end{equation}
\begin{equation}\label{eq:normal H1 continuity-ad-df}
\mathbf{u}\cdot\mathbf{n}\Lbrack c\Rbrack=0\quad\forall\Lambda\in\mathcal{F}_h^{\text{\rm ad-df},0,+},
\end{equation}
\begin{equation}\label{eq:normal H1 continuity-ad-ad}
\mathbf{u}\cdot\mathbf{n}\Lbrack c\Rbrack=0\quad\forall\Lambda\in\mathcal{F}_h^{\text{\rm ad-ad},0}.
\end{equation}
\end{subequations}
Then
\begin{equation}\label{eq:consistency}
\mathscr{L}_{\text{\rm dfd},h}^S(c,v_h)=\mathscr{R}_h^T(v_h)
\quad\forall v_h\in U_h.
\end{equation}
\end{theorem}
\begin{proof}

First note that if $\mathbf{n}_{\partial K}(\mathbf{x})^\top A(\mathbf{x})\mathbf{n}_{\partial K}(\mathbf{x})=0$ a.e. $\mathbf{x}\in\bar{\Lambda}\subset\partial K$, where $\mathbf{n}_{\partial K}(\mathbf{x})$ is the outward unit normal to $K$, then   
\begin{equation}\label{eq:An0}
\begin{aligned}
&A(\mathbf{x})\mathbf{n}_{\partial K}(\mathbf{x})=0 \text{ a.e. } \mathbf{x}\in\bar{\Lambda}. 
\end{aligned}
\end{equation}

Second note that \eqref{eq:normal continuity} 
gives 
\begin{equation}\label{eq:normal continuity-C}
\Lbrack (-A\nabla c+\mathbf{u}c)\cdot\mathbf{n}_\Lambda\Rbrack=0\quad\forall \Lambda\in \mathcal{F}_h^0,
\end{equation}
while \eqref{Eq:Boundary essential and natrual condition} gives 
\begin{equation*}
(-A\nabla c+\mathbf{u}c)\cdot\mathbf{n}|_{\Gamma_{\text{df},N,-}}=\chi^{-}, \quad -A\nabla c\cdot{\bf n}|_{\Gamma_{\text{df},N,+}}=\chi^{+}.
\end{equation*}
Moreover, \eqref{eq:normal H1 continuity-ad-ad} coincides with \eqref{eq:normal continuity-C} on $\mathcal{F}_h^{\text{\rm ad-ad},0}$.

Then, by the integration by parts, the conclusion \eqref{eq:consistency} follows from \eqref{Eq:Bilinear form minus}, \eqref{Eq:Linear form}, \eqref{Eq:BVP} or \eqref{Eq:Equivalent PDE}, \eqref{Eq:Boundary essential and natrual condition}, \eqref{eq:advection Hdiv}, \eqref{eq:normal continuity-C}, \eqref{eq:H1 continuity}, \eqref{eq:normal H1 continuity-ad-df},  \eqref{eq:normal H1 continuity-ad-ad}, \eqref{eq:An0}, and \eqref{eq:correct boundary subset}. 
The verification is detailed below, with two main steps. 

\textit{Step 1.} In the right side of \eqref{eq:consistency} (see \eqref{Eq:Linear form}), 
\begin{equation*}
\begin{aligned}
& (g,v_h)_{L^2(\mathcal{T}_h)}=\sum_{K\in\mathcal{T}_h}(-\iv (A\nabla c)+\frac{\mathbf{u}\cdot\nabla c}{2}+\frac{\iv(\mathbf{u}c)}{2}+(\gamma  +\frac{\iv\mathbf{u}}{2})c,v_h)_{L^2(K)},
\end{aligned}
\end{equation*}
where, by the integration by parts, we have
\begin{equation*}
\begin{aligned}
&\sum_{K\in\mathcal{T}_h}(-\iv (A\nabla c),v_h)_{L^2(K)}=\sum_{K\in\mathcal{T}_h}(A\nabla c),\nabla v_h)_{L^2(K)}+\sum_{K\in\mathcal{T}_h}(- (A\nabla c)\cdot\mathbf{n},v_h)_{L^2(\partial K)},
\end{aligned}
\end{equation*}
where 
\begin{equation*}
\begin{aligned}
\sum_{K\in\mathcal{T}_h}(- (A\nabla c)\cdot\mathbf{n},v_h)_{L^2(\partial K)}=&\sum_{\Lambda\in\mathcal{F}_h^0\cup\mathcal{F}_{\Gamma,h}}(\Lbr - (A\nabla c)\cdot\mathbf{n}\Rbr,\Lbrack v_h\Rbrack )_{L^2(\Lambda)}
+\sum_{\Lambda\in\mathcal{F}_h^0}(\Lbrack- (A\nabla c)\cdot\mathbf{n}\Rbrack,\Lbr v_h\Rbr)_{L^2(\Lambda)},
\end{aligned}
\end{equation*}
Because, on the advection boundary $\Gamma_\text{ad}$,  
\begin{equation*}
\sum_{\Lambda\subset\Gamma_{\text{ad}} 
}(\Lbr - (A\nabla c)\cdot\mathbf{n}\Rbr,\Lbrack v_h\Rbrack )_{L^2(\Lambda)}=\sum_{\Lambda\subset\Gamma_\text{ad} 
}(- (A\nabla c)\cdot\mathbf{n},v_h)_{L^2(\Lambda)}
=0,
\end{equation*}
we have
\begin{equation*}
\begin{aligned}
&\sum_{\Lambda\in\mathcal{F}_h^0\cup\mathcal{F}_{\Gamma,h}}(\Lbr - (A\nabla c)\cdot\mathbf{n}\Rbr,\Lbrack v_h\Rbrack )_{L^2(\Lambda)}=
\\
&\sum_{\Lambda\in\mathcal{F}_h^0\cup\mathcal{F}_{\text{df},D,h}
}(\Lbr - (A\nabla c)\cdot\mathbf{n}\Rbr,\Lbrack v_h\Rbrack )_{L^2(\Lambda)}
+\sum_{\Lambda\subset\Gamma_{\text{df},N} 
}(- (A\nabla c)\cdot\mathbf{n}, v_h )_{L^2(\Lambda)}
=
\\
& \sum_{\Lambda\in\mathcal{F}_h^0\cup\mathcal{F}_{\text{df},D,h}}(\Lbr - (A\nabla c)\cdot\mathbf{n}\Rbr,\Lbrack v_h\Rbrack )_{L^2(\Lambda)}
\\
&+
(\chi^-,v_h )_{L^2(\Gamma_{\text{df},N,-})} 
+
(-\mathbf{u}\cdot\mathbf{n}c,v_h )_{L^2(\Gamma_{\text{df},N,-})} 
+
(\chi^+,v_h )_{L^2(\Gamma_{\text{df},N,+})} 
.
\end{aligned}
\end{equation*}
Because
\begin{equation*}
\sum_{\Lambda\in\mathcal{F}_h^{\text{ad-ad},0}}(\Lbr - (A\nabla c)\cdot\mathbf{n}\Rbr,\Lbrack v_h\Rbrack )_{L^2(\Lambda)}=0,
\end{equation*}
we have 
\begin{equation*}
\begin{aligned}
&  \sum_{\Lambda\in\mathcal{F}_h^0}(\Lbr - (A\nabla c)\cdot\mathbf{n}\Rbr,\Lbrack v_h\Rbrack )_{L^2(\Lambda)}= \sum_{\Lambda\in\mathcal{F}_h^{\text{df-df},0}}(\Lbr - (A\nabla c)\cdot\mathbf{n}\Rbr,\Lbrack v_h\Rbrack )_{L^2(\Lambda)}
\\
&+\sum_{\Lambda\in\mathcal{F}_h^{\text{ad-df},0,+}}(\Lbr - (A\nabla c)\cdot\mathbf{n}\Rbr,\Lbrack v_h\Rbrack )_{L^2(\Lambda)}+\sum_{\Lambda\in\mathcal{F}_h^{\text{ad-df},0,-}}(\Lbr - (A\nabla c)\cdot\mathbf{n}\Rbr,\Lbrack v_h\Rbrack )_{L^2(\Lambda)}.
\end{aligned}
\end{equation*}
Because
\begin{equation*}
\Lbrack (-A\nabla c+\mathbf{u}c)\cdot\mathbf{n}\Rbrack=(-A\nabla c)\cdot\mathbf{n}+\mathbf{u}\cdot\mathbf{n} \Lbrack c\Rbrack=0\quad\mbox{on $\Lambda\in\mathcal{F}_h^{\text{ad-df},0}$},
\end{equation*}
we have 
\begin{equation*}
\begin{aligned}
\sum_{\Lambda\in\mathcal{F}_h^{\text{ad-df},0,+}}(\Lbr - (A\nabla c)\cdot\mathbf{n}\Rbr,\Lbrack v_h\Rbrack )_{L^2(\Lambda)}=&\frac{1}{2}\sum_{\Lambda\in\mathcal{F}_h^{\text{ad-df},0,+}}(- (A\nabla c)\cdot\mathbf{n},\Lbrack v_h\Rbrack )_{L^2(\Lambda)}
\\
=& \frac{1}{2}\sum_{\Lambda\in\mathcal{F}_h^{\text{ad-df},0,+}}(- \mathbf{u}\cdot\mathbf{n} \Lbrack c\Rbrack,\Lbrack v_h\Rbrack )_{L^2(\Lambda)}
\\
=& 0\quad\mbox{(because $\mathbf{u}\cdot\mathbf{n} \Lbrack c\Rbrack=0$ on $\mathcal{F}_h^{\text{ad-df},0,+}$)}, 
\end{aligned}
\end{equation*}
\begin{equation*}
\begin{aligned}
\sum_{\Lambda\in\mathcal{F}_h^{\text{ad-df},0,-}}(\Lbr - (A\nabla c)\cdot\mathbf{n}\Rbr,\Lbrack v_h\Rbrack )_{L^2(\Lambda)}=&\frac{1}{2}\sum_{\Lambda\in\mathcal{F}_h^{\text{ad-df},0,-}}(- (A\nabla c)\cdot\mathbf{n},\Lbrack v_h\Rbrack )_{L^2(\Lambda)}
\\
=& 
\frac{1}{2}\sum_{\Lambda\in\mathcal{F}_h^{\text{ad-df},0,-}}(- \mathbf{u}\cdot\mathbf{n} \Lbrack c\Rbrack,\Lbrack v_h\Rbrack )_{L^2(\Lambda)}. 
\end{aligned}
\end{equation*}

Similar to the term $\sum_{\Lambda\in\mathcal{F}_h^0\cup\mathcal{F}_{\Gamma,h}}(\Lbr - (A\nabla c)\cdot\mathbf{n}\Rbr,\Lbrack v_h\Rbrack )_{L^2(\Lambda)}$, we have 
\begin{equation*}
\begin{aligned}
& \sum_{\Lambda\in\mathcal{F}_h^0}(\Lbrack- (A\nabla c)\cdot\mathbf{n}\Rbrack,\Lbr v_h\Rbr)_{L^2(\Lambda)}= \sum_{\Lambda\in\mathcal{F}_h^{\text{df-df},0}}(\Lbrack - (A\nabla c)\cdot\mathbf{n}\Rbrack,\Lbr v_h\Rbr)_{L^2(\Lambda)}
\\
&+\sum_{\Lambda\in\mathcal{F}_h^{\text{ad-df},0,+}}(\Lbrack - (A\nabla c)\cdot\mathbf{n}\Rbrack,\Lbr v_h\Rbr )_{L^2(\Lambda)}+\sum_{\Lambda\in\mathcal{F}_h^{\text{ad-df},0,-}}(\Lbrack - (A\nabla c)\cdot\mathbf{n}\Rbrack,\Lbr v_h\Rbr)_{L^2(\Lambda)}
\\
&=\sum_{\Lambda\in\mathcal{F}_h^{\text{df-df},0}}(- \mathbf{u}\cdot\mathbf{n} \Lbrack c\Rbrack,\Lbr v_h\Rbr)_{L^2(\Lambda)}
+\sum_{\Lambda\in\mathcal{F}_h^{\text{ad-df},0,+}}(- \mathbf{u}\cdot\mathbf{n}\Lbrack c\Rbrack,\Lbr v_h\Rbr )_{L^2(\Lambda)}
\\
&+\sum_{\Lambda\in\mathcal{F}_h^{\text{ad-df},0,-}}(-\mathbf{u}\cdot\mathbf{n} \Lbrack c\Rbrack,\Lbr v_h\Rbr)_{L^2(\Lambda)}
=\sum_{\Lambda\in\mathcal{F}_h^{\text{ad-df},0,-}}(-\mathbf{u}\cdot\mathbf{n} \Lbrack c\Rbrack,\Lbr v_h\Rbr)_{L^2(\Lambda)}. 
\end{aligned}
\end{equation*}
But
\begin{equation*}
\begin{aligned}
& 
\frac{1}{2}\sum_{\Lambda\in\mathcal{F}_h^{\text{ad-df},0,-}}(- \mathbf{u}\cdot\mathbf{n} \Lbrack c\Rbrack,\Lbrack v_h\Rbrack )_{L^2(\Lambda)}
+\sum_{\Lambda\in\mathcal{F}_h^{\text{ad-df},0,-}}(-\mathbf{u}\cdot\mathbf{n} \Lbrack c\Rbrack,\Lbr v_h\Rbr)_{L^2(\Lambda)}
=
\\
&  
\frac{1}{2}\sum_{\substack{\Lambda\in\mathcal{F}_h^{\text{ad-df},0,-}\\ \Lambda\subset\partial_{-}K \\ \Lambda=\partial_{-}K\cap\partial_{+}K_1}}
(- \mathbf{u}\cdot\mathbf{n} \Lbrack c\Rbrack,(v_h|_K-v_h|_{K_1}) )_{L^2(\Lambda)
}
+\frac{1}{2}\sum_{\substack{\Lambda\in\mathcal{F}_h^{\text{ad-df},0,-}\\ \Lambda\subset\partial_{-}K \\ \Lambda=\partial_{-} K\cap \partial_{+} K_1}} 
(-\mathbf{u}\cdot\mathbf{n} \Lbrack c\Rbrack,(v_h|_K+v_h|_{K_1}))_{L^2(\Lambda)
}
\\
& = 
- 
(\mathbf{u}\cdot\mathbf{n} \Lbrack c\Rbrack,v_h)_{L^2(\Sigma_{\text{ad-df},0,-})}, 
\end{aligned}
\end{equation*}
where $\Sigma_{\text{ad-df},0,-}$ denotes the subset of $\Omega$, which is defined by   
$\Sigma_{\text{ad-df},0,-}:=\{\mathbf{x}\in\Omega: \mathbf{x}\in\Lambda \text{ for some } \Lambda\subset \mathcal{F}_h^{\text{ad-df},0,-} 
\}$, 
i.e., $\Sigma_{\text{ad-df},0,-}$ is the set of all points $\mathbf{x}\in\Omega$ belonging to $\mathcal{F}_h^{\text{ad-df},0,-}$. 

\textit{Step 2.} In the left side of \eqref{eq:consistency} (see \eqref{Eq:Bilinear form-1}), with $\Gamma_{\text{df},D,-}$ given by \eqref{eq:inflow df-D} and define the set  
$\Sigma_0:=\{\mathbf{x}\in\Omega: \mathbf{x}\in\Lambda \text{ for some } \Lambda\subset\mathcal{F}_h^0\}$, 
we have 
\begin{equation*}
\begin{aligned}
-\sum\limits_{K\in{\cal T}_h} \langle v_h, {\bf u}\cdot{\bf n} \Lbrack c\Rbrack \rangle_{L^2(\partial_- K)}=&- 
\langle v_h, {\bf u}\cdot{\bf n} c \rangle_{L^2(\Gamma_{\text{ad},-}\cup\Gamma_{\text{df},D,-})} 
-
(\mathbf{u}\cdot\mathbf{n}c,v_h )_{L^2(\Gamma_{\text{df},N,-})}
\\
&- 
\langle v_h, {\bf u}\cdot{\bf n} \Lbrack c\Rbrack \rangle_{L^2(\Sigma_0)}, 
\end{aligned}
\end{equation*}
where  
$c|_{\Gamma_{\text{ad},-}\cup\Gamma_{\text{df},D,-}}=\kappa$.

Because 
\begin{equation*}
\begin{aligned}
&\mathcal{F}_h^0=\mathcal{F}_h^{\text{df-df},0}\cup\mathcal{F}_h^{\text{ad},0,+}\cup\mathcal{F}_h^{\text{ad},0,-}\cup\mathcal{F}_h^{\text{ad-ad},0}, 
\\
& \Lbrack c\Rbrack|_{\mathcal{F}_h^{\text{df-df},0}} =0,
\\
& \mathbf{u}\cdot\mathbf{n}\Lbrack c\Rbrack|_{\mathcal{F}_h^{\text{ad},0,+}\cup \mathcal{F}_h^{\text{ad-ad},0}}=0, 
\end{aligned}
\end{equation*}
we have 
\begin{equation*}
-\langle v_h, {\bf u}\cdot{\bf n} \Lbrack c\Rbrack \rangle_{L^2(\Sigma_0)}=
- (\mathbf{u}\cdot\mathbf{n} \Lbrack c\Rbrack,v_h)_{L^2(\Sigma_{\text{ad-df},0,-})}.  
\end{equation*}
For the three artificial terms, we have  
\begin{equation*}
\begin{aligned}
\sum_{\Lambda\in\mathcal{F}_h^{\text{df-df},0}\cup \mathcal{F}_{\text{df},D,h}} \langle \Lbr (A\nabla v_h)\cdot\mathbf{n}_\Lambda\Rbr, \Lbrack c\Rbrack\rangle_{L^2(\Lambda)}
=&\sum_{\Lambda\in\mathcal{F}_{\text{df},D,h}} \langle (A\nabla v_h)\cdot\mathbf{n}, \kappa\rangle_{L^2(\Lambda)},
\end{aligned}
\end{equation*}
\begin{equation*}
\begin{aligned}
\sum_{\Lambda\in\mathcal{F}_h^{\text{df-df-dfd},0}} |\Lambda|^{-\frac{1}{d-1}}
A_\Lambda \langle\Lbrack c\Rbrack, \Lbrack v_h\Rbrack\rangle_{L^2(\Lambda)}
+\sum_{\Lambda\in\mathcal{F}_{\text{df-dfd},D,h}} |\Lambda|^{-\frac{1}{d-1}}
A_\Lambda \langle c, v_h\rangle_{L^2(\Lambda)}
=&
\\
\sum_{\Lambda\in\mathcal{F}_{\text{df-dfd},D,h}} |\Lambda|^{-\frac{1}{d-1}}
A_\Lambda \langle \kappa, v_h\rangle_{L^2(\Lambda)},& 
\\
S_h(c,v_h)=&T_h(v_h).
\end{aligned}
\end{equation*}
Combining all the above equalities in the two steps, the proof is then complete. 
\end{proof}

\begin{remark}
The assumptions \eqref{eq:H1 continuity}, \eqref{eq:normal H1 continuity-ad-df}, \eqref{eq:normal H1 continuity-ad-ad} show that $c$ is continuous only across the diffusion-diffusion interelement boundaries, while it can be discontinuous elsewhere. Note that \eqref{eq:normal H1 continuity-ad-df} does not necessarily imply the continuity of the solution; if $\mathbf{u}\cdot\mathbf{n}\not=0$ almost everywhere along $\Lambda$, then can we conclude that $c$ is continuous there.  While \eqref{eq:normal H1 continuity-ad-ad} is well-known in the advection-reaction equation (i.e., $A=0$). Thus, these assumptions are realistic. 
\end{remark}

\begin{remark}\label{rem:example}
In \cite{DiPietroDroniouErn2015} and \cite{DiPietroErnGuermond2008}, a case of only an interface across which the nature of the equation changes is considered, where the continuity condition, instead of  \eqref{eq:normal H1 continuity-ad-df},  is given by $\Lbrack c\Rbrack=0$. Such continuity is not correct unless $\mathbf{u}(\mathbf{x})\cdot\mathbf{n}(\mathbf{x})\not=0$ a.e. $\mathbf{x}$ on the interface. In fact, there is an example (\cite[Example 4, page 2159]{HoustonSchwabSuli2002}) for which the solution $\Lbrack c\Rbrack\not=0$ everywhere along the interface but $\mathbf{u}\cdot\mathbf{n}\Lbrack c\Rbrack=0$ everywhere along the interface. Therefore, the correct continuity condition is \eqref{eq:normal H1 continuity-ad-df} here. Such example had  not been covered by any known DG method (including the few references \cite{HoustonSchwabSuli2002}, \cite{ProftRiviere2009}, \cite{DiPietroDroniouErn2015}, \cite{DiPietroErnGuermond2008}).  
\end{remark}

Next, we investigate the SUPG-type stability. 

\begin{lemma}\label{lem:convectionreconstructionproperty}
There are properties of the advection: 
\begin{itemize}
\item {\it The advection coercivity}: for all $v\in \prod_{K\in\mathcal{T}_h} H^1(K)$, 

\begin{equation}\label{Eq:Coercivity of convection}
\begin{aligned}
&\sum\limits_{K\in{\cal T}_h}(\frac{\bf u}{2}\cdot\nabla v+\frac{\iv({\bf u}v)}{2},v)_{L^2(K)}-\sum\limits_{K\in{\cal T}_h} \langle v, {\bf u}\cdot{\bf n} \Lbrack v\Rbrack \rangle_{L^2(\partial_- K)}
\\
&=\dfrac{1}{2}\sum\limits_{\Lambda\in{\cal F}_h^0} \langle |{\bf u}\cdot{\bf n}_\Lambda|,|\Lbrack v\Rbrack|^2\rangle_{L^2(\Lambda)}
+\frac{1}{2} \sum\limits_{\Lambda\in\mathcal{F}_{\Gamma,h}}
\langle |{\bf u}\cdot{\bf n}|,|v|^2\rangle_{L^2(\Lambda)}.
\end{aligned}
\end{equation}

\item {\it The advection dual property}: for all $v,w\in \prod_{K\in\mathcal{T}_h} H^1(K)$,  

\begin{equation}\label{Eq:Convection for interpolaiton error estimates}
\begin{aligned}
&\sum\limits_{K\in{\cal T}_h}(\frac{\bf u}{2}\cdot\nabla v+\frac{\iv({\bf u}v)}{2},w)_{L^2(K)}-\sum\limits_{K\in{\cal T}_h} \langle w, {\bf u}\cdot{\bf n} \Lbrack v\Rbrack \rangle_{L^2(\partial_- K)}
\\
&=-\sum\limits_{K\in{\cal T}_h}(v,\frac{\iv({\bf u}w)}{2}+\frac{\bf u}{2}\cdot\nabla w)_{L^2(K)}
+  \sum\limits_{K\in{\cal T}_h} \langle v, {\bf u}\cdot{\bf n} \Lbrack w\Rbrack \rangle_{L^2(\partial_+ K)}.
\end{aligned}
\end{equation}
\end{itemize} 
\end{lemma}

\begin{proof}
All these properties follow from the integration by parts. 
\end{proof}

From \eqref{Eq:Coercivity of convection} in Lemma \ref{lem:convectionreconstructionproperty}, 
we first define the energy norm,  
from the bilinear form $\mathscr{L}_{\text{\rm dfd},h}^S(\cdot,\cdot)$ which is given by \eqref{Eq:Bilinear form-1},  
\begin{equation}\label{eq:DG norm}
\begin{aligned}
\|v\|_{\mathscr{L}_{\text{\rm dfd},h}^S}^2:=&\mathscr{L}_{\text{\rm dfd},h}^S(v,v)
=\|A^\frac{1}{2}\nabla v\|_{L^2(\mathcal{T}_h)}^2
\\
&+\sum_{
\Lambda\in\mathcal{F}_h^{\text{df-df-dfd},0} 
} |\Lambda|^{-\frac{1}{d-1}} 
A_{\Lambda} \|\Lbrack v\Rbrack\|_{L^2(\Lambda)}^2
+\sum_{
\Lambda\in\mathcal{F}_{\text{df-dfd},D,h} 
} |\Lambda|^{-\frac{1}{d-1}}
A_{\Lambda}\|v\|_{L^2(\Lambda)}^2
\\
&+ \dfrac{1}{2}\sum\limits_{\Lambda\in{\cal F}_h^0} \langle |{\bf u}\cdot{\bf n}_\Lambda|,|\Lbrack v\Rbrack|^2\rangle_{L^2(\Lambda)}
+\frac{1}{2} \sum_{\Lambda\in\mathcal{F}_{\Gamma,h}}\langle |{\bf u}\cdot{\bf n}|,|v|^2\rangle_{L^2(\Lambda)}
\\
&+\|\gamma_0^\frac{1}{2}v\|_{L^2(\mathcal{T}_h)}^2+S_h(v,v).
\end{aligned}
\end{equation}
\begin{lemma}
There exists a positive constant $C_1$ such that when $\alpha$ satisfies  
\begin{equation}\label{eq:alpha}
0<\alpha<\min\Big(\frac{2}{3},C_1\Big),
\end{equation}
then the SUPG-type coercivity holds: 
\begin{equation*}
\begin{aligned}
& \|v_h\|_{\mathscr{L}_{\text{\rm dfd},h}^S}^2\ge \frac{1}{2} \|\tau_K^\frac{1}{2} (\frac{\bf u}{2}\cdot\nabla v_h+\frac{\iv({\bf u}v_h)}{2})\|_{L^2(\mathcal{T}_h)}^2\quad\forall 
v_h\in U_h.
\end{aligned}
\end{equation*}
\end{lemma}
\begin{proof}
We need only to estimate $S_h(v_h,v_h)$. From the definition of $S_h(\cdot,\cdot)$ (see \eqref{eq:stabilization}), we have 
\begin{equation*}
\begin{aligned}
S_h(v_h,v_h)=& -\sum_{K\in\mathcal{T}_h}(\iv(A\nabla v_h), \tau_K (\frac{\bf u}{2}\cdot\nabla v_h+\frac{\iv({\bf u}v_h)}{2}))_{L^2(K)}+ \sum_{K\in\mathcal{T}_h}(\iv(A\nabla v_h), \tau_K \gamma_0 v_h)_{L^2(K)}
\\
&+\|\tau_K^\frac{1}{2} (\frac{\bf u}{2}\cdot\nabla v_h+\frac{\iv({\bf u}v_h)}{2})\|_{L^2(\mathcal{T}_h)}^2-\|\tau_K^\frac{1}{2} \gamma_0 v_h\|_{L^2(\mathcal{T}_h)}^2. 
\end{aligned}
\end{equation*}
By the local inverse inequality, letting $C_\text{inv}$ denote the corresponding constant, we have 
\begin{equation*}
h_K \|\iv(A\nabla v_h)\|_{L^2(K)}\le C_\text{inv} \|A\nabla v_h\|_{L^2(K)}\le C_\text{inv} A_{\max,K}^\frac{1}{2}\|A^\frac{1}{2}\nabla v_h\|_{L^2(K)}.  
\end{equation*}
Then, by the definition \eqref{eq:stabilizing parameter pointwise} of $\tau_K$,  
\begin{equation*}
\|\tau_K^\frac{1}{2}\iv(A\nabla v_h)\|_{L^2(K)}\le \big(\alpha\big)^\frac{1}{2} C_\text{inv}\|A^\frac{1}{2}\nabla v_h\|_{L^2(K)}.
\end{equation*}
Hence, by Cauchy-Schwarz inequality and Young's inequality $|a||b|\le \delta |a|^2+\frac{|b|^2}{4\delta}$ with $\delta=\frac{1}{2}$, 
\begin{equation*}
\begin{aligned}
&-\sum_{K\in\mathcal{T}_h}(\iv(A\nabla v_h), \tau_K (\frac{\bf u}{2}\cdot\nabla v_h+\frac{\iv({\bf u}v_h)}{2}))_{L^2(K)}\ge 
\\
&-\sum_{K\in\mathcal{T}_h}\|\tau_K^\frac{1}{2}\iv(A\nabla v_h)\|_{L^2(K)}\|\tau_K^\frac{1}{2} (\frac{\bf u}{2}\cdot\nabla v_h+\frac{\iv({\bf u}v_h)}{2})\|_{L^2(K)}
\\
&\ge -\frac{1}{4\delta}\sum_{K\in\mathcal{T}_h}\|\tau_K^\frac{1}{2}\iv(A\nabla v_h)\|_{L^2(K)}^2-\delta \sum_{K\in\mathcal{T}_h}\|\tau_K^\frac{1}{2} (\frac{\bf u}{2}\cdot\nabla v_h+\frac{\iv({\bf u}v_h)}{2})\|_{L^2(K)}^2
\\
&\ge -\frac{\alpha (C_\text{inv})^2}{4\delta}\sum_{K\in\mathcal{T}_h}\|A^\frac{1}{2}\nabla v_h\|_{L^2(K)}^2-\delta \sum_{K\in\mathcal{T}_h}\|\tau_K^\frac{1}{2} (\frac{\bf u}{2}\cdot\nabla v_h+\frac{\iv({\bf u}v_h)}{2})\|_{L^2(K)}^2
\\
&=-\frac{\alpha (C_\text{inv})^2}{2}
\|A^\frac{1}{2}\nabla v_h\|_{L^2(\mathcal{T}_h)}^2-\frac{1}{2} \sum_{K\in\mathcal{T}_h}\|\tau_K^\frac{1}{2} (\frac{\bf u}{2}\cdot\nabla v_h+\frac{\iv({\bf u}v_h)}{2})\|_{L^2(K)}^2
\end{aligned}
\end{equation*}
and
\begin{equation*}
\begin{aligned}
&\sum_{K\in\mathcal{T}_h}(\iv(A\nabla v_h), \tau_K \gamma_0 v_h)_{L^2(K)}\ge 
 -\frac{\alpha (C_\text{inv})^2}{2} 
 \|A^\frac{1}{2}\nabla v_h\|_{L^2(\mathcal{T}_h)}^2-\frac{1}{2} \|\tau_K^\frac{1}{2}\gamma_0v_h\|_{L^2(\mathcal{T}_h)}^2. 
\end{aligned}
\end{equation*}
In addition, 
\begin{equation}
\|\tau_K^\frac{1}{2}\gamma_0v_h\|_{L^2(K)}^2=\int_K \tau_K (\gamma_0 v_h)^2 \le \alpha \int_K |\gamma_0| v_h^2=\alpha \|\gamma_0^\frac{1}{2} v_h\|_{L^2(K)}^2.  
\end{equation}
Combining all the above gives 
\begin{equation*}
\begin{aligned}
S_h(v_h,v_h)\ge & \frac{1}{2} \|\tau_K^\frac{1}{2} (\frac{\bf u}{2}\cdot\nabla v_h+\frac{\iv({\bf u}v_h)}{2})\|_{L^2(\mathcal{T}_h)}^2
-\alpha (C_\text{inv})^2 
\|A^\frac{1}{2}\nabla v_h\|_{L^2(\mathcal{T}_h)}^2-\frac{3\alpha}{2} \|\gamma_0^\frac{1}{2} v_h\|_{L^2(\mathcal{T}_h)}^2.
\end{aligned}
\end{equation*}
Therefore, with $C_1:=\frac{1}{\big(C_\text{inv}\big)^2}$ and $\alpha$ given by \eqref{eq:alpha}, the conclusion follows. 
\end{proof}

\section{Error estimates}
\label{sec:Error estimates}

Let $c$ denote the exact solution. Let $\pi_hc\in U_h$ be the interpolation, $\pi_h c|_K:=\pi_K c$, where $\pi_K$ satisfies \eqref{Eq:interpolation error-0}, \eqref{Eq:interpolation error-1} in \textit{Assumption 5}. 

We assume that the exact solution $c$ has the following regularity for a real number $\hat{\ell}$ 
\begin{equation}
c|_K\in W^{1+\hat{\ell},4}(K),\quad \frac{1}{4}<\hat{\ell}\le\ell. 
\end{equation}
The condition $\hat{\ell}>\frac{1}{4}$ is invoked from the local trace inequality for dealing with the $\nabla c$ term in the error estimates. 
When $0<\hat{\ell}\le\frac{1}{4}$, the error estimates, associating with the $\nabla c$ term, obtained in what follows 
can still hold,  
with some sophisticated argument like the one in 
\cite{CaiyeZhang2011} applied (such argument is rather involved). 

For the error estimates, we introduce a parameter about the diffusion coefficient $A$, the advection velocity $\mathbf{u}$, the reaction coefficient $\gamma$ and the element diameter $h_K$ as follows:     
\begin{equation}\label{Eq:Stabilizing parameter-1}
\begin{aligned}
\mathcal{D}_K(A,{\bf u},\gamma,h_K):= 
& \varepsilon_{\text{df},K} + 
 h_K^2 |K|^{-1} 
\|{\bf u}\cdot{\bf n}\|_{L^1(\partial K)} +h_K |K|^{-\frac{1}{2}}  
\|{\bf u}\|_{L^2(K)}
\\
&+h_K^2 |K|^{-\frac{1}{2}} \|\iv\mathbf{u}\|_{L^2(K)}+h_K^2 |K|^{-\frac{1}{2}}\|\gamma_0\|_{L^2(K)},
\end{aligned}
\end{equation}
where $\gamma_0=\gamma+\frac{\iv\mathbf{u}}{2}$ and ${\bf u}\in \{{\bf v}\in \prod_{K\in{\cal T}_h} H(\iv;K): {\bf v}\cdot{\bf n}|_{\partial K}\in L^1(\partial K),\forall K\in{\cal T}_h\}$. It is used only for the theoretical purpose and does not enter into the DG method and the implementation.

Let $v_h:=\pi_hc-c_h$. Then 
\begin{equation}\label{eq:error equation-0}
\|v_h\|_{\mathscr{L}_{\text{dfd},h}^S}^2=\mathscr{L}_{\text{\rm dfd},h}^S(v_h,v_h), 
\end{equation}
where, from \eqref{eq:consistency} in Theorem \ref{thm:consistency}, 
\begin{equation}\label{eq:error estimates}
\begin{aligned}
&\mathscr{L}_{\text{dfd},h}^S(v_h,v_h)=\mathscr{L}_{\text{dfd},h}^S(\pi_hc-c_h,v_h)=\mathscr{L}_{\text{dfd},h}^S(\pi_hc-c,v_h), 
\end{aligned}
\end{equation}
where
\begin{equation}\label{eq:interpolation error estimates}
\begin{aligned}
&\mathscr{L}_{\text{dfd},h}^S(\pi_hc-c,v_h)=
\\
&\sum_{K\in\mathcal{T}_h} (A\nabla (\pi_hc- c),\nabla v_h)_{L^2(K)}-\sum_{\Lambda\in\mathcal{F}_h^{\text{df-df},0}\cup \mathcal{F}_{\text{df},D,h}} \langle \Lbr (A\nabla (\pi_h c-c))\cdot\mathbf{n}_\Lambda\Rbr, \Lbrack v_h\Rbrack\rangle_{L^2(\Lambda)} 
\\
&+\sum_{\Lambda\in\mathcal{F}_h^{\text{df-df},0}\cup \mathcal{F}_{\text{df},D,h}} \langle \Lbr (A\nabla v_h)\cdot\mathbf{n}_\Lambda\Rbr, \Lbrack \pi_hc-c\Rbrack\rangle_{L^2(\Lambda)}
\\
&+ \sum_{\Lambda\in\mathcal{F}_h^{\text{df-df-dfd},0}} |\Lambda|^{-\frac{1}{d-1}} 
A_{\Lambda} \langle\Lbrack \pi_hc-c\Rbrack, \Lbrack v_h\Rbrack\rangle_{L^2(\Lambda)}
+ \sum_{\Lambda\in\mathcal{F}_{\text{df-dfd},D,h}} |\Lambda|^{-\frac{1}{d-1}} 
A_{\Lambda} \langle \pi_hc-c, v_h\rangle_{L^2(\Lambda)}
\\
&+\sum\limits_{K\in{\cal T}_h}(\frac{\bf u}{2}\cdot\nabla (\pi_hc-c)+\frac{\iv({\bf u}(\pi_hc-c))}{2},v_h)_{L^2(K)}-\sum\limits_{K\in{\cal T}_h} \langle v_h, {\bf u}\cdot{\bf n} \Lbrack \pi_hc-c\Rbrack \rangle_{L^2(\partial_- K)}
\\
&+(\gamma_0 (\pi_hc-c),v_h)_{L^2({\cal T}_h)}+S_h(\pi_h c-c,v_h).
\end{aligned}
\end{equation}

We estimate the nine terms one-by-one in the right of \eqref{eq:interpolation error estimates}.  

The first term is estimated as follows:  
\begin{equation}\label{eq:first term error}
\begin{aligned}
\sum_{K\in\mathcal{T}_h} (A\nabla (\pi_hc- c),\nabla v_h)_{L^2(K)}\le & C \left(\sum_{K\in\mathcal{T}_h} A_{\max,K} h_K^{2\min(\ell,\hat{\ell})} \|c\|_{H^{1+\hat{\ell}}(K)}^2\right)^\frac{1}{2}
\left(\sum_{K\in\mathcal{T}_h} \|A^\frac{1}{2}\nabla v_h\|_{L^2(K)}^2\right)^\frac{1}{2}.
\end{aligned}
\end{equation}

The second term is estimated as follows: 

\begin{equation}\label{eq:second term error}
\begin{aligned}
&\sum_{\Lambda\in\mathcal{F}_h^{\text{df-df},0}\cup \mathcal{F}_{\text{df},D,h}} \langle \Lbr (A\nabla (\pi_h c-c))\cdot\mathbf{n}_\Lambda\Rbr, \Lbrack v_h\Rbrack\rangle_{L^2(\Lambda)}=
\\
& \sum_{\Lambda\in\mathcal{F}_h^{\text{df-df-dfd},0}\cup \mathcal{F}_{\text{df-dfd},D,h}} \langle \Lbr (A\nabla (\pi_h c-c))\cdot\mathbf{n}_\Lambda\Rbr, \Lbrack v_h\Rbrack\rangle_{L^2(\Lambda)} 
+\sum_{\Lambda\in\mathcal{F}_h^{\text{df-df-add},0}\cup \mathcal{F}_{\text{df-add},D,h}} \langle \Lbr (A\nabla (\pi_h c-c))\cdot\mathbf{n}_\Lambda\Rbr, \Lbrack v_h\Rbrack\rangle_{L^2(\Lambda)}.
\end{aligned}
\end{equation}
In what follows, we estimate the above two subterms of \eqref{eq:second term error}. 

If $\hat{\ell}>\frac{1}{2}$, we estimate the first subterm as follows: by the local trace inequality \eqref{eq:local trace inequality}, 
\begin{equation}\label{eq:second term error-1}
\begin{aligned}
&\sum_{\Lambda\in\mathcal{F}_h^{\text{df-df-dfd},0}\cup \mathcal{F}_{\text{df-dfd},D,h}} \langle \Lbr (A\nabla (\pi_h c-c))\cdot\mathbf{n}_\Lambda\Rbr, \Lbrack v_h\Rbrack\rangle_{L^2(\Lambda)}\le
\\
& \sum_{\Lambda\in\mathcal{F}_h^{\text{df-df-dfd},0}\cup \mathcal{F}_{\text{df-dfd},D,h}} \sum_{i=1}^2 \|A\nabla (\pi_h c-c)|_{K_i}\|_{\Lambda}\|\Lbrack v_h\Rbrack\|_{L^2(\Lambda)}
\\
&\le C \left(\left(\sum_{\substack{\Lambda\in\mathcal{F}_h^{\text{df-df-dfd},0}\\ \Lambda=\partial K_1\cap\partial K_2}} \sum_{i=1}^2|\Lambda|^\frac{1}{d-1} \frac{|\Lambda|}{|K_i|} A_{\max,\Lambda,\partial K_i} (|\pi_hc-c|_{H^1(K_i)}^2+h_K^{2\hat{\ell}}|\pi_hc-c|_{H^{1+\hat{\ell}}(K_i)}^2)\right)^\frac{1}{2}\right.
\\
&+\left. \left(\sum_{\substack{\Lambda\in\mathcal{F}_{\text{df-dfd},D,h}\\ \Lambda=\partial K\cap\Gamma_{\text{df},D}}} |\Lambda|^\frac{1}{d-1} \frac{|\Lambda|}{|K|} A_{\max,\Lambda,\partial K} (|\pi_hc-c|_{H^1(K)}^2+h_K^{2\hat{\ell}}|\pi_hc-c|_{H^{1+\hat{\ell}}(K)}^2)\right)^\frac{1}{2}\right)
\\
&\times \left(\sum_{\Lambda\in\mathcal{F}_h^{\text{df-df-dfd},0}\cup \mathcal{F}_{\text{df-dfd},D,h}} |\Lambda|^{-\frac{1}{d-1}} A_\Lambda \|\Lbrack v_h\Rbrack\|_{L^2(\Lambda)}^2\right)^\frac{1}{2}
\\
&\le C  \left(\sum_{\substack{\Lambda\in\mathcal{F}_h^{\text{df-df-dfd},0}\\ \Lambda=\partial K_1\cap\partial K_2}} \sum_{i=1}^2 A_{\max,\Lambda,\partial K_i} h_{K_i}^{2\min(\ell,\hat{\ell})} \|c\|_{H^{1+\hat{\ell}}(K_i)}^2\right.
\\
&\left.+\sum_{\Lambda\in\mathcal{F}_{\text{df-dfd},D,h}}A_{\max,\Lambda,\partial K} h_{K}^{2\min(\ell,\hat{\ell})} \|c\|_{H^{1+\hat{\ell}}(K)}^2\right)^\frac{1}{2}
\\
&\times \left(\sum_{\Lambda\in\mathcal{F}_h^{\text{df-df-dfd},0}\cup \mathcal{F}_{\text{df-dfd},D,h}} |\Lambda|^{-\frac{1}{d-1}} A_\Lambda \|\Lbrack v_h\Rbrack\|_{L^2(\Lambda)}^2\right)^\frac{1}{2},
\end{aligned}
\end{equation}
where we have used the following facts: the interpolation property \eqref{Eq:interpolation error-0} in \textit{Assumption A5} and the shape-regularity \eqref{eq:shape-regularity} in \textit{Assumption 1}  
\begin{equation}\label{eq:shape-regularity-Lambda}
|\Lambda|^\frac{1}{d-1} \frac{|\Lambda|}{|K|}\le C \frac{|\partial K| h_\Lambda}{|K|}\le C \frac{|\partial K| h_K}{|K|}
\le C. 
\end{equation}
Regarding the second subterm, noticing the diffusion-diffusion with advection-dominated element boundaries in \eqref{eq:advection-dominated df-df-0} and \eqref{eq:advection-dominated df-df-1}, following the argument for proving the first subterm,  we have
\begin{equation}\label{eq:second term error-2}
\begin{aligned}
& \sum_{\Lambda\in\mathcal{F}_h^{\text{df-df-add},0}\cup \mathcal{F}_{\text{df-add},D,h}} \langle \Lbr (A\nabla (\pi_h c-c))\cdot\mathbf{n}_\Lambda\Rbr, \Lbrack v_h\Rbrack\rangle_{L^2(\Lambda)}\le
\\
&\le C  \left(\sum_{\substack{\Lambda\in\mathcal{F}_h^{\text{df-df-add},0}\\ \Lambda=\partial K_1\cap\partial K_2}} \sum_{i=1}^2 A_{\max,\Lambda,\partial K_i} h_{K_i}^{2\min(\ell,\hat{\ell})} \|c\|_{H^{1+\hat{\ell}}(K_i)}^2\right.
\\\
&\left.+\sum_{\Lambda\in\mathcal{F}_{\text{df-add},D,h}}A_{\max,\Lambda,\partial K} h_{K}^{2\min(\ell,\hat{\ell})} \|c\|_{H^{1+\hat{\ell}}(K)}^2\right)^\frac{1}{2}
\\
\\
&\times \left(\dfrac{1}{2}\sum\limits_{\Lambda\in{\cal F}_h^0} \langle |{\bf u}\cdot{\bf n}_\Lambda|,|\Lbrack v_h\Rbrack|^2\rangle_{L^2(\Lambda)}
+\frac{1}{2} \sum_{\Lambda\in\mathcal{F}_{\Gamma,h}}\langle |{\bf u}\cdot{\bf n}|,|v_h|^2\rangle_{L^2(\Lambda)}\right)^\frac{1}{2}.
\end{aligned}
\end{equation}

If $\frac{1}{4}<\hat{\ell}\le\frac{1}{2}$, we use instead the local trace inequality \eqref{eq:local trace ienquality wsp}, similarly as \eqref{eq:second term error-1} and \eqref{eq:second term error-2}, to obtain the estimates of the two subterms of 
\eqref{eq:second term error}, respectively, as follows:  
\begin{equation}\label{eq:second term error-11}
\begin{aligned}
&\sum_{\Lambda\in\mathcal{F}_h^{\text{df-df-dfd},0}\cup \mathcal{F}_{\text{df-dfd},D,h}} \langle \Lbr (A\nabla (\pi_h c-c))\cdot\mathbf{n}_\Lambda\Rbr, \Lbrack v_h\Rbrack\rangle_{L^2(\Lambda)}\le
\\
&\le C  \left(\sum_{\substack{\Lambda\in\mathcal{F}_h^{\text{df-df-dfd},0}\\ \Lambda=\partial K_1\cap\partial K_2}} \sum_{i=1}^2 A_{\max,\Lambda,\partial K_i} h_{K_i}^{2\min(\ell,\hat{\ell})} (\|c\|_{H^{1+\hat{\ell}}(K_i)}^2+|K_i|^\frac{1}{2}\|c\|_{W^{1+\hat{\ell},4}(K_i)}^2)\right.
\\
&\left.+\sum_{\Lambda\in\mathcal{F}_{\text{df-dfd},D,h}}A_{\max,\Lambda,\partial K} h_{K}^{2\min(\ell,\hat{\ell})} (\|c\|_{H^{1+\hat{\ell}}(K)}^2+|K|^\frac{1}{2}\|c\|_{W^{1+\hat{\ell},4}(K)}^2)\right)^\frac{1}{2}
\\
&\times \left(\sum_{\Lambda\in\mathcal{F}_h^{\text{df-df-dfd},0}\cup \mathcal{F}_{\text{df-dfd},D,h}} |\Lambda|^{-\frac{1}{d-1}} A_\Lambda \|\Lbrack v_h\Rbrack\|_{L^2(\Lambda)}^2\right)^\frac{1}{2},
\end{aligned}
\end{equation}

\begin{equation}\label{eq:second term error-22}
\begin{aligned}
& \sum_{\Lambda\in\mathcal{F}_h^{\text{df-df-add},0}\cup \mathcal{F}_{\text{df-add},D,h}} \langle \Lbr (A\nabla (\pi_h c-c))\cdot\mathbf{n}_\Lambda\Rbr, \Lbrack v_h\Rbrack\rangle_{L^2(\Lambda)}\le
\\
&\le C  \left(\sum_{\substack{\Lambda\in\mathcal{F}_h^{\text{df-df-add},0}\\ \Lambda=\partial K_1\cap\partial K_2}} \sum_{i=1}^2 A_{\max,\Lambda,\partial K_i} h_{K_i}^{2\min(\ell,\hat{\ell})} (\|c\|_{H^{1+\hat{\ell}}(K_i)}^2+|K_i|^\frac{1}{2}\|c\|_{W^{1+\hat{\ell},4}(K_i)}^2)\right.
\\
&\left. +\sum_{\Lambda\in\mathcal{F}_{\text{df-add},D,h}}A_{\max,\Lambda,\partial K} h_{K}^{2\min(\ell,\hat{\ell})} (\|c\|_{H^{1+\hat{\ell}}(K)}^2+|K|^\frac{1}{2}\|c\|_{W^{1+\hat{\ell},4}(K)}^2)\right)^\frac{1}{2}
\\
&\times \left(\dfrac{1}{2}\sum\limits_{\Lambda\in{\cal F}_h^0} \langle |{\bf u}\cdot{\bf n}_\Lambda|,|\Lbrack v_h\Rbrack|^2\rangle_{L^2(\Lambda)}
+\frac{1}{2} \sum_{\Lambda\in\mathcal{F}_{\Gamma,h}}\langle |{\bf u}\cdot{\bf n}|,|v_h|^2\rangle_{L^2(\Lambda)}\right)^\frac{1}{2}.
\end{aligned}
\end{equation}

The third term is estimated as follows: 

\begin{equation*}
\begin{aligned}
& \sum_{\Lambda\in\mathcal{F}_h^{\text{df-df},0}\cup \mathcal{F}_{\text{df},D,h}} \langle \Lbr (A\nabla v_h)\cdot\mathbf{n}_\Lambda\Rbr, \Lbrack \pi_hc-c\Rbrack\rangle_{L^2(\Lambda)}\le
\\
& C \sum_{\substack{\Lambda\in\mathcal{F}_h^{\text{df-df},0}\\ \Lambda=\partial K_1\cap\partial K_2}} \sum_{i=1}^2 \|(A^\frac{1}{2}\nabla v_h)|_{K_i}\|_{L^2(\Lambda)} \sum_{i=1}^2 A_\Lambda^\frac{1}{2}\|(\pi_hc-c)|_{K_i}\|_{L^2(\Lambda)}
\\
&+C \sum_{\substack{\Lambda\in\mathcal{F}_{\text{df},D,h}\\ \Lambda=\partial K\cap \Gamma_{\text{df},D}}} \|(A^\frac{1}{2}\nabla v_h)|_{K}\|_{L^2(\Lambda)}  A_\Lambda^\frac{1}{2}\|(\pi_hc-c)|_{K}\|_{L^2(\Lambda)}
\\
&=C \sum_{\substack{\Lambda\in\mathcal{F}_h^{\text{df-df},0}\\ \Lambda=\partial K_1\cap\partial K_2}} |\Lambda|^\frac{1}{2(d-1)}\sum_{i=1}^2 \|(A^\frac{1}{2}\nabla v_h)|_{K_i}\|_{L^2(\Lambda)} |\Lambda|^{-\frac{1}{2(d-1)}}A_\Lambda^\frac{1}{2}\sum_{i=1}^2 \|(\pi_hc-c)|_{K_i}\|_{L^2(\Lambda)}
\\
&+ C \sum_{\substack{\Lambda\in\mathcal{F}_{\text{df},D,h}\\ \Lambda=\partial K\cap \Gamma_{\text{df},D}}} |\Lambda|^\frac{1}{2(d-1)}\|(A^\frac{1}{2}\nabla v_h)|_{K}\|_{L^2(\Lambda)} |\Lambda|^{-\frac{1}{2(d-1)}}A_\Lambda^\frac{1}{2} \|(\pi_hc-c)|_{K}\|_{L^2(\Lambda)}, 
\end{aligned}
\end{equation*}
where, by the local trace inequality \eqref{eq:local trace inequality} on the finite-dimensional space and the bound \eqref{eq:shape-regularity-Lambda}, 
\begin{equation*}
\begin{aligned}
&|\Lambda|^\frac{1}{d-1}\sum_{i=1}^2 \|(A^\frac{1}{2}\nabla v_h)|_{K_i}\|_{L^2(\Lambda)}^2=
\\
&|\Lambda|^\frac{1}{d-1} \frac{|\Lambda|}{|K_1|} \frac{|K_1|}{|\Lambda|}\|(A^\frac{1}{2}\nabla v_h)|_{K_1}\|_{L^2(\Lambda)}^2+|\Lambda|^\frac{1}{d-1} \frac{|\Lambda|}{|K_2|} \frac{|K_2|}{|\Lambda|}\|(A^\frac{1}{2}\nabla v_h)|_{K_2}\|_{L^2(\Lambda)}^2 
\\
&\le C \sum_{i=1}^2\|A^\frac{1}{2}\nabla v_h\|_{L^2(K_i)}^2 
\end{aligned}
\end{equation*}
and
\begin{equation*}
|\Lambda|^\frac{1}{d-1} \|(A^\frac{1}{2}\nabla v_h)|_{K}\|_{L^2(\Lambda)}^2\le C \|A^\frac{1}{2}\nabla v_h\|_{L^2(K)}^2, 
\end{equation*}
while, by \eqref{eq:nonconformity}  and \eqref{Eq:interpolation error-0}, 
\begin{equation}\label{eq:third term error-1}
\begin{aligned}
& |\Lambda|^{-\frac{1}{d-1}}A_\Lambda\sum_{i=1}^2 \|(\pi_hc-c)|_{K_i}\|_{L^2(\Lambda)}^2= 
\\
& |\Lambda|^{-\frac{1}{d-1}}A_\Lambda \frac{|\Lambda|}{|K_1|} \frac{|K_1|}{|\Lambda|}\|(\pi_hc-c)|_{K_1}\|_{L^2(\Lambda)}^2+|\Lambda|^{-\frac{1}{d-1}}A_\Lambda \frac{|\Lambda|}{|K_2|} \frac{|K_2|}{|\Lambda|}\|(\pi_hc-c)|_{K_2}\|_{L^2(\Lambda)}^2
\\
& \le C \sum_{i=1}^2 A_\Lambda h_{K_i}^{2\min(\ell,\hat{\ell})} \|c\|_{H^{1+\hat{\ell}}(K_i)}^2
\end{aligned}
\end{equation}
and
\begin{equation}\label{eq:third term error-2}
|\Lambda|^{-\frac{1}{d-1}}A_\Lambda\|(\pi_hc-c)|_{K}\|_{L^2(\Lambda)}^2\le C A_\Lambda h_{K}^{2\min(\ell,\hat{\ell})} \|c\|_{H^{1+\hat{\ell}}(K)}^2,
\end{equation}
and consequently, the error estimates of the third term are as follows: 
\begin{equation}\label{eq:third term error}
\begin{aligned}
& \sum_{\Lambda\in\mathcal{F}_h^{\text{df-df},0}\cup \mathcal{F}_{\text{df},D,h}} \langle \Lbr (A\nabla v_h)\cdot\mathbf{n}_\Lambda\Rbr, \Lbrack \pi_hc-c\Rbrack\rangle_{L^2(\Lambda)}\le
\\
&  C \left(\sum_{\substack{\Lambda\in\mathcal{F}_h^{\text{df-df},0}\\ \Lambda=\partial K_1\cap\partial K_2}} \sum_{i=1}^2 A_\Lambda h_{K_i}^{2\min(\ell,\hat{\ell})} \|c\|_{H^{1+\hat{\ell}}(K_i)}^2+\sum_{\substack{\Lambda\in\mathcal{F}_{\text{df},D,h}\\ \Lambda=\partial K\cap \Gamma_{\text{df},D}}} A_\Lambda h_{K}^{2\min(\ell,\hat{\ell})} \|c\|_{H^{1+\hat{\ell}}(K)}^2\right)^\frac{1}{2}
\\
&\times \left(\sum_{K\in\mathcal{T}_h} \|A^\frac{1}{2}\nabla v_h\|_{L^2(K)}^2\right)^\frac{1}{2}.
\end{aligned}
\end{equation}

The fourth and fifth terms can be estimated the same as \eqref{eq:third term error-1} and \eqref{eq:third term error-2}, and we have 
\begin{equation}\label{eq:four-five error}
\begin{aligned}
&\sum_{\Lambda\in\mathcal{F}_h^{\text{df-df-dfd},0}} |\Lambda|^{-\frac{1}{d-1}}A_{\Lambda} \langle\Lbrack \pi_hc-c\Rbrack, \Lbrack v_h\Rbrack\rangle_{L^2(\Lambda)}+ \sum_{\Lambda\in\mathcal{F}_{\text{df-dfd},D,h}} |\Lambda|^{-\frac{1}{d-1}} 
A_{\Lambda} \langle \pi_hc-c, v_h\rangle_{L^2(\Lambda)}\le
\\
& C \left(\sum_{\substack{\Lambda\in\mathcal{F}_h^{\text{df-df-dfd},0}\\ \Lambda=\partial K_1\cap\partial K_2}} \sum_{i=1}^2 A_\Lambda h_{K_i}^{2\min(\ell,\hat{\ell})} \|c\|_{H^{1+\hat{\ell}}(K_i)}^2+\sum_{\substack{\Lambda\in\mathcal{F}_{\text{df-dfd},D,h}\\ \Lambda=\partial K\cap \Gamma_{\text{df},D}}} A_\Lambda h_{K}^{2\min(\ell,\hat{\ell})} \|c\|_{H^{1+\hat{\ell}}(K)}^2\right)^\frac{1}{2}
\\
&\times \left( \sum_{\Lambda\in\mathcal{F}_h^{\text{df-df-dfd},0}} |\Lambda|^{-\frac{1}{d-1}}A_{\Lambda} \|\Lbrack v_h\Rbrack\|_{L^2(\Lambda)}^2+ \sum_{\Lambda\in\mathcal{F}_{\text{df-dfd},D,h}} |\Lambda|^{-\frac{1}{d-1}} 
A_{\Lambda} \|v_h\|_{L^2(\Lambda)}^2\right)^\frac{1}{2}.
\end{aligned}
\end{equation}

The sixth and seventh terms must be estimated altogether. By the advection dual property \eqref{Eq:Convection for interpolaiton error estimates} in Lemma \ref{lem:convectionreconstructionproperty}, we have

\begin{equation*}
\begin{array}{l}
\sum\limits_{K\in{\cal T}_h}(\frac{\bf u}{2}\cdot\nabla (\pi_hc-c)+\frac{\iv({\bf u}(\pi_hc-c))}{2},v_h)_{L^2(K)}-\sum\limits_{K\in{\cal T}_h} \langle v_h, {\bf u}\cdot{\bf n} \Lbrack \pi_hc-c\Rbrack \rangle_{L^2(\partial_- K)}
\\
=-\sum\limits_{K\in{\cal T}_h}(\pi_h c-c,\frac{\iv({\bf u}v_h)}{2}+\frac{\bf u}{2}\cdot\nabla v_h)_{L^2(K)}
+  \sum\limits_{K\in{\cal T}_h} \langle \pi_h c-c, {\bf u}\cdot{\bf n} \Lbrack v_h\Rbrack \rangle_{L^2(\partial_+ K)}.
\end{array}
\end{equation*}
We estimate the above two right subterms one-by-one. The first subterm is estimated as follows: 
\begin{equation*}
\begin{aligned}
& (c-\pi_hc,\frac{\iv({\bf u}v_h)}{2}+\frac{\bf u}{2}\cdot\nabla v_h)_{L^2(K)}=(\tau_K^{-\frac{1}{2}}(c-\pi_hc),\tau_K^\frac{1}{2} (\frac{\iv({\bf u}v_h)}{2}+\frac{\bf u}{2}\cdot\nabla v_h))_{L^2(K)},
\end{aligned}
\end{equation*}
where, by the definition of $\tau_K$ in \eqref{eq:stabilizing parameter pointwise} and the definition of $\mathcal{D}_K(A,\mathbf{u},\gamma,h_K)$ in  \eqref{Eq:Stabilizing parameter-1}, and by the interpolation error estimates \eqref{Eq:interpolation error-1} in \textit{Assumption 5}, 
\begin{equation*}
\begin{aligned}
&\|\tau_K^{-\frac{1}{2}}(c-\pi_h c)\|_{L^2(K)}^2=\int_K \tau_K^{-1} (c-\pi_hc)^2
\\
&=\int_K(\varepsilon_{\text{df},K}+h_K |\mathbf{u}|_2+h_K^2 |\gamma_0|)h_K^{-2} (c-\pi_hc)^2
\\
&\le C \left( \varepsilon_{\text{df},K}^2 |K|+h_K^2 \|\mathbf{u}\|_{L^2(K)}^2+h_K^4 \|\gamma_0\|_{L^2(K)}^2\right)^\frac{1}{2} \left(h_K^{-4} \int_K (c-\pi_hc)^4\right)^\frac{1}{2} 
\\
&= C \left( \varepsilon_{\text{df},K}^2 |K|+h_K^2 \|\mathbf{u}\|_{L^2(K)}^2+h_K^4 \|\gamma_0\|_{L^2(K)}^2\right)^\frac{1}{2} \left(h_K^{-2} \|c-\pi_hc\|_{L^4(K)}^2\right)
\\
&\le C h_K^{2\min(\ell,\hat{\ell})}  \left( \varepsilon_{\text{df},K} |K|^\frac{1}{2}+h_K \|\mathbf{u}\|_{L^2(K)}+h_K^2 \|\gamma_0\|_{L^2(K)}\right)\|c\|_{W^{1+\hat{\ell},4}(K)}^2
\\
&=C h_K^{2\min(\ell,\hat{\ell})}|K|^\frac{1}{2} \mathcal{D}_K(A,\mathbf{u},\gamma,h_K) \|c\|_{W^{1+\hat{\ell},4}(K)}^2.
\end{aligned}
\end{equation*}
The second subterm is estimated as follows: 
\begin{equation*}
\begin{aligned}
&\langle c-\pi_hc, {\bf u}\cdot{\bf n} \Lbrack v_h\Rbrack \rangle_{L^2(\partial_+ K)}\le \langle |c-\pi_hc|, |{\bf u}\cdot{\bf n}| |\Lbrack v_h\Rbrack| \rangle_{L^2(\partial K)}
\\
&\le \sum\limits_{\Lambda\subset\partial K} \|(c-\pi_hc)|{\bf u}\cdot{\bf n}|^\frac{1}{2}\|_{L^2(\Lambda)}\, \| |{\bf u}\cdot{\bf n}|^\frac{1}{2}\Lbrack v_h\Rbrack\|_{L^2(\Lambda)}
\\
&\le \left(\sum\limits_{\Lambda\subset\partial K} \|(c-\pi_hc)|{\bf u}\cdot{\bf n}|^\frac{1}{2}\|_{L^2(\Lambda)}^2\right)^\frac{1}{2} 
\left(\sum\limits_{\Lambda\subset\partial K}\|\,|{\bf u}\cdot{\bf n}|^\frac{1}{2}\Lbrack v_h\Rbrack\|_{L^2(\Lambda)}^2\right)^\frac{1}{2},
\end{aligned}
\end{equation*}
where, noticing that the Sobolev imbedding $W^{1,4}\hookrightarrow L^\infty$,  by the interpolation error estimates \eqref{Eq:interpolation error-1}, 
\begin{equation*}
\begin{aligned}
\|(c-\pi_hc)|{\bf u}\cdot{\bf n}|^\frac{1}{2}\|_{L^2(\Lambda)}^2=&\int_\Lambda |{\bf u}\cdot{\bf n}|\,(c-\pi_h c)^2
\\
\le & \|{\bf u}\cdot{\bf n}\|_{L^1(\Lambda)}\|c-\pi_hc\|_{L^\infty(\Lambda)}^2
\\
\le & \|{\bf u}\cdot{\bf n}\|_{L^1(\Lambda)}\|c-\pi_hc\|_{L^\infty(K)}^2
\\
\le & C |K|^{-\frac{1}{2}} h_K^{2(1+\min(\ell,\hat{\ell}))} \|c\|_{W^{1+\hat{\ell},4}(K)}^2 \|{\bf u}\cdot{\bf n}\|_{L^1(\Lambda)},  
\end{aligned}
\end{equation*}
and as a result, 
\begin{equation*}
\begin{aligned}
\left(\sum\limits_{\Lambda\subset\partial K} \|(c-\pi_hc)|{\bf u}\cdot{\bf n}|^\frac{1}{2}\|_{L^2(\Lambda)}^2\right)^\frac{1}{2}\le & C \left(\sum\limits_{\Lambda\subset\partial K} |K|^{-\frac{1}{2}} h_K^{2(1+\min(\ell,\hat{\ell})} \|c\|_{W^{1+\hat{\ell},4}(K)}^2 \|{\bf u}\cdot{\bf n}\|_{L^1(\Lambda)}\right)^\frac{1}{2}
\\
=& C |K|^\frac{1}{4} h_K^{\min(\ell,\hat{\ell})} \|c\|_{W^{1+\hat{\ell},4}(K)} |K|^{-\frac{1}{2}} h_K \|\mathbf{u}\cdot\mathbf{n}\|_{L^1(\partial K)}^\frac{1}{2} 
\\
\le & C \Big(\mathcal{D}_K(A,\mathbf{u},\gamma,h_K)\Big)^\frac{1}{2} h_K^{\min(\ell,\hat{\ell})} |K|^\frac{1}{4} \|c\|_{W^{1+\hat{\ell},4}(K)}. 
\end{aligned}
\end{equation*}
Hence, on each $K$, 
\begin{equation*}
\begin{aligned}
\langle c-\pi_hc, {\bf u}\cdot{\bf n} \Lbrack v_h\Rbrack \rangle_{\partial_+ K}\le  & C
h_K^{\min(\ell,\hat{\ell})}|K|^\frac{1}{4} \Big({\cal D}_K(A,{\bf u},\gamma,h_K)\Big)^\frac{1}{2} \|c\|_{W^{1+\hat{\ell},4}(K)}
\\
& \times \left(\sum\limits_{\Lambda\subset\partial K}\|\,|{\bf u}\cdot{\bf n}|^{1/2}\Lbrack v_h\Rbrack\|_{L^2(\Lambda)}^2\right)^\frac{1}{2}.
\end{aligned}
\end{equation*}\
Summing over all $K\in\mathcal{T}_h$ gives the estimates of the sixth and seventh terms as follows:  
\begin{equation}\label{eq:six-seven error}
\begin{aligned}
&\sum\limits_{K\in{\cal T}_h}(\frac{\bf u}{2}\cdot\nabla (\pi_hc-c)+\frac{\iv({\bf u}(\pi_hc-c))}{2},v_h)_{L^2(K)}-\sum\limits_{K\in{\cal T}_h} \langle v_h, {\bf u}\cdot{\bf n} \Lbrack \pi_hc-c\Rbrack \rangle_{L^2(\partial_- K)}\le
\\
& C \left(\sum_{K\in\mathcal{T}_h} h_K^{2 \min(\ell,\hat{\ell})} |K|^\frac{1}{2} {\cal D}_K(A,{\bf u},\gamma,h_K)\|c\|_{W^{1+\hat{\ell},4}(K)}^2\right)^\frac{1}{2}
\\
&\times \left(\dfrac{1}{2}\sum\limits_{\Lambda\in{\cal F}_h^0} \langle |{\bf u}\cdot{\bf n}_\Lambda|,|\Lbrack v_h\Rbrack|^2\rangle_{L^2(\Lambda)}
+\frac{1}{2} \sum_{\Lambda\in\mathcal{F}_{\Gamma,h}}\langle |{\bf u}\cdot{\bf n}|,|v_h|^2\rangle_{L^2(\Lambda)}\right)^\frac{1}{2}.
\end{aligned}
\end{equation}

The eighth term of the reaction term is estimated as follows:  
\begin{equation*}
(\gamma_0 (c-\pi_h c),v_h)_{L^2(K)}\le \|\gamma_0^\frac{1}{2}(c-\pi_h c)\|_{L^2(K)}\,\|\gamma_0^\frac{1}{2}v_h\|_{L^2(K)},
\end{equation*}
\begin{equation*}
\begin{aligned}
\|\gamma_0^\frac{1}{2}(c-\pi_h c)\|_{L^2(K)}^2=&\int_K |\gamma_0| (c-\pi_h c)^2
\\
\le & \|\gamma_0\|_{L^2(K)} \|c-\pi_h c\|_{L^4(K)}^2
\\
\le & C h_K^{2+2\min(\ell,\hat{\ell})} \|c\|_{W^{1+\hat{\ell},4}(K)}^2 \|\gamma_0\|_{L^2(K)},
\end{aligned}
\end{equation*}
and hence
\begin{equation*}
\begin{aligned}
(\gamma_0 (c-\pi_h c),v_h)_{L^2(K)} &\le C
h_K^{\min(\ell,\hat{\ell})}|K|^\frac{1}{4} \Big({\cal D}_K(A,{\bf u},\gamma,h_K)\Big)^{1/2} \|c\|_{W^{1+\hat{\ell},4}(K)}\,\|\gamma_0^\frac{1}{2}v_h\|_{L^2(K)}.
\end{aligned}
\end{equation*}
Summing over $\mathcal{T}_h$ gives the estimates of the eighth term as follows: 
\begin{equation}\label{eq:eight error}
(\gamma_0 (c-\pi_h c),v_h)_{L^2(\mathcal{T}_h)} \le C \left(\sum_{K\in\mathcal{T}_h} h_K^{2 \min(\ell,\hat{\ell})} |K|^\frac{1}{2} {\cal D}_K(A,{\bf u},\gamma,h_K)\|c\|_{W^{1+\hat{\ell},4}(K)}^2\right)^\frac{1}{2}
\|\gamma_0^\frac{1}{2}v_h\|_{L^2(\mathcal{T}_h)}.
\end{equation} 

The last stabilization term is estimated as follows: 
\begin{equation*}
S_h(c-\pi_h c,v_h)\le \Big(S_h(c-\pi_hc,c-\pi_hc)\Big)^\frac{1}{2} \Big(S_h(v_h,v_h)\Big)^\frac{1}{2}.
\end{equation*}
In $\Big(S_h(c-\pi_hc,c-\pi_hc)\Big)^\frac{1}{2}$, on each $K\in\mathcal{T}_h$ (see \eqref{eq:stabilization}), we only estimate 
\begin{equation}\label{eq:stabilization error-0}
\|\tau_K^\frac{1}{2} (\frac{\iv({\bf u}(c-\pi_h c))}{2}+\frac{\bf u}{2}\cdot\nabla (c-\pi_hc))\|_{L^2(K)}
\le  \|\tau_K^\frac{1}{2}(c-\pi_hc)\frac{\iv{\bf u}}{2}\|_{L^2(K)}+\|\tau_K^\frac{1}{2}{\bf u}\cdot\nabla (c-\pi_hc)\|_{L^2(K)},
\end{equation}
while the term 
\[
\|\tau_K^\frac{1}{2}\iv(A\nabla (c-\pi_hc)\|_{L^2(K)}
\]
can be estimated easily, giving 
\begin{equation}\label{eq:stabilization error-1}
\|\tau_K^\frac{1}{2}\iv(A\nabla (c-\pi_hc)\|_{L^2(K)}\le C h_K^{\min(\ell,\hat{\ell})} (A_{\max, K}^{-\frac{1}{2}}\|A\nabla c\|_{H^{\hat{\ell}}(K)}+A_{\max,K}^\frac{1}{2}\|c\|_{H^{1+\hat{\ell}}(K)})
\end{equation} 
and the term 
\[
\|\tau_K^\frac{1}{2}\gamma_0(c-\pi_hc)\|_{L^2(K)}
\]
can be similarly estimated as the previous reaction term (see \eqref{eq:eight error}), giving    
\begin{equation}\label{eq:stabilization error-2}
\begin{aligned}
\|\tau_K^\frac{1}{2}\gamma_0(c-\pi_hc)\|_{L^2(K)}\le & C \|\gamma_0^\frac{1}{2}(c-\pi_hc)\|_{L^2(K)}
\\
\le & C h_K^{\min(\ell,\hat{\ell})}|K|^\frac{1}{4} \Big({\cal D}_K(A,{\bf u},\gamma,h_K)\Big)^{1/2} \|c\|_{W^{1+\hat{\ell},4}(K)}. 
\end{aligned}
\end{equation}
By 
the definition of $\tau_K$ in \eqref{eq:stabilizing parameter pointwise} and the definition of $\mathcal{D}_K(A,\mathbf{u},\gamma,h_K)$ in \eqref{Eq:Stabilizing parameter-1}, 
for the first subterm in the right of \eqref{eq:stabilization error-0}, noticing that    
\begin{equation*}
\begin{aligned}
 \|\tau_K^\frac{1}{2}(c-\pi_hc)\frac{\iv{\bf u}}{2}\|_{L^2(K)} &\le C \||\iv{\bf u}|^\frac{1}{2}(c-\pi_h c)\|_{L^2(K)}, 
 \end{aligned}
 \end{equation*}
 similar as \eqref{eq:stabilization error-2}, we have 
\begin{equation*}
\|\tau_K^\frac{1}{2}(c-\pi_hc)\frac{\iv{\bf u}}{2}\|_{L^2(K)}\le Ch_K^{\min(\ell,\hat{\ell})}|K|^\frac{1}{4} \Big({\cal D}_K(A,{\bf u},\gamma,h_K)\Big)^{1/2} \|c\|_{W^{1+\hat{\ell},4}(K)}. 
\end{equation*}
Estimating the second subterm in the right of  \eqref{eq:stabilization error-0},  we have  
\begin{equation*}
\begin{aligned}
\|\tau_K^\frac{1}{2}{\bf u}\cdot\nabla (c-\pi_hc)\|_{L^2(K)} &\le h_K^\frac{1}{2}\| |{\bf u}|^\frac{1}{2}\nabla(c-\pi_hc)\|_{L^2(K)}
\\
&\le C h_K^\frac{1}{2}\|\mathbf{u}\|_{L^2(K)}^\frac{1}{2}|c-\pi_hc|_{W^{1,4}(K)}
\\
&\le C h_K^{\min(\ell,\hat{\ell})} |K|^\frac{1}{4} \Big(\mathcal{D}_K(A,\mathbf{u},\gamma,h_K)\Big)^\frac{1}{2} |c|_{W^{1+\hat{\ell},4}(K)}.
\end{aligned}
\end{equation*}
Hence, we have the estimates of the stabilization term as follows: 
\begin{equation}\label{eq:nine error}
S_h(c-\pi_h c,v_h)\le C \left(\sum_{K\in\mathcal{T}_h} h_K^{2 \min(\ell,\hat{\ell})} |K|^\frac{1}{2} {\cal D}_K(A,{\bf u},\gamma,h_K)\|c\|_{W^{1+\hat{\ell},4}(K)}^2\right)^\frac{1}{2}\Big(S_h(v_h,v_h)\Big)^\frac{1}{2}.
\end{equation}
Summing up all the above estimates: \eqref{eq:error equation-0}, \eqref{eq:error estimates}, \eqref{eq:interpolation error estimates}, \eqref{eq:first term error}, \eqref{eq:second term error} with \eqref{eq:second term error-1} and \eqref{eq:second term error-2} or with \eqref{eq:second term error-11} and \eqref{eq:second term error-22},  \eqref{eq:third term error}, \eqref{eq:four-five error}, \eqref{eq:six-seven error}, \eqref{eq:eight error}, and \eqref{eq:nine error}, we conclude the error estimates in the following theorem.

\begin{theorem}\label{thm:error estimates-thm}
Let $c$ be the exact solution, having the regularity $W^{1+\hat{\ell},4}(K)$ on each $K$, with a real number $\hat{\ell}$ satisfying $\frac{1}{4}<\hat{\ell}\le \ell$, and $c_h,\pi_hc\in U_h$ be the DG solution and the DG interpolation, respectively. Then
\begin{equation*}
\begin{aligned}
& \|\pi_hc-c_h\|_{\mathscr{L}_{\text{\rm dfd},h}^S}\le C \left( \sum_{K\in\mathcal{T}_h} (A_{\max,K} h_K^{2\min(\ell,\hat{\ell})} \|c\|_{H^{1+\hat{\ell}}(K)}^2+A_{\max,K}^{-1} h_K^{2\min(\ell,\hat{\ell})}\|A\nabla c\|_{H^{\hat{\ell}}(K)}^2)\right)^\frac{1}{2} 
\\
&+C  \left(\sum_{\substack{\Lambda\in\mathcal{F}_h^{\text{\rm df-df-add},0}\\ \Lambda=\partial K_1\cap\partial K_2}} \sum_{i=1}^2 A_{\max,\Lambda,\partial K_i} h_{K_i}^{2\min(\ell,\hat{\ell})} (\|c\|_{H^{1+\hat{\ell}}(K_i)}^2+|K_i|^\frac{1}{2}\|c\|_{W^{1+\hat{\ell},4}(K_i)}^2)\right.
\\
&\qquad\qquad\left. +\sum_{\Lambda\in\mathcal{F}_{\text{\rm df-add},D,h}}A_{\max,\Lambda,\partial K} h_{K}^{2\min(\ell,\hat{\ell})} (\|c\|_{H^{1+\hat{\ell}}(K)}^2+|K|^\frac{1}{2}\|c\|_{W^{1+\hat{\ell},4}(K)}^2)\right)^\frac{1}{2}
\\
&  C \left(\sum_{\substack{\Lambda\in\mathcal{F}_h^{\text{\rm df-df},0}\\ \Lambda=\partial K_1\cap\partial K_2}} \sum_{i=1}^2 A_\Lambda h_{K_i}^{2\min(\ell,\hat{\ell})} \|c\|_{H^{1+\hat{\ell}}(K_i)}^2+\sum_{\substack{\Lambda\in\mathcal{F}_{\text{\rm df},D,h}\\ \Lambda=\partial K\cap \Gamma_{\text{\rm df},D}}} A_\Lambda h_{K}^{2\min(\ell,\hat{\ell})} \|c\|_{H^{1+\hat{\ell}}(K)}^2\right)^\frac{1}{2}
\\
&+C \left(\sum\limits_{K\in{\cal T}_h} h_K^{2\min(\ell,\hat{\ell})}|K|^\frac{1}{2} {\cal D}_K(A,{\bf u},\gamma,h_K) \|c\|_{W^{1+\hat{\ell},4}(K)}^2\right)^\frac{1}{2}.
\end{aligned}
\end{equation*}
\end{theorem}

By the triangle inequality, 
\[
\|c-c_h\|_{\mathscr{L}_{\text{\rm dfd},h}^S}\le \|c-\pi_hc\|_{\mathscr{L}_{\text{\rm dfd},h}^S}+\|\pi_hc-c_h\|_{\mathscr{L}_{\text{\rm dfd},h}^S},
\]
where $\|c-\pi_hc\|_{\mathscr{L}_{\text{\rm dfd},h}^S}$ can be estimated similarly as $\|\pi_hc-c_h\|_{\mathscr{L}_{\text{\rm dfd},h}^S}$ so that we obtain the following conclusion: 
\begin{corollary}
Under the sam conditions as in Theorem \ref{thm:error estimates-thm}, we have  
\begin{equation*}
\begin{aligned}
& \|c-c_h\|_{\mathscr{L}_{\text{\rm dfd},h}^S}\le
C \left( \sum_{K\in\mathcal{T}_h} (A_{\max,K} h_K^{2\min(\ell,\hat{\ell})} \|c\|_{H^{1+\hat{\ell}}(K)}^2+A_{\max,K}^{-1} h_K^{2\min(\ell,\hat{\ell})}\|A\nabla c\|_{H^{\hat{\ell}}(K)}^2)\right)^\frac{1}{2} 
\\
&+C  \left(\sum_{\substack{\Lambda\in\mathcal{F}_h^{\text{\rm df-df-add},0}\\ \Lambda=\partial K_1\cap\partial K_2}} \sum_{i=1}^2 A_{\max,\Lambda,\partial K_i} h_{K_i}^{2\min(\ell,\hat{\ell})} (\|c\|_{H^{1+\hat{\ell}}(K_i)}^2+|K_i|^\frac{1}{2}\|c\|_{W^{1+\hat{\ell},4}(K_i)}^2)\right.
\\
&\qquad\qquad\left. +\sum_{\Lambda\in\mathcal{F}_{\text{\rm df-add},D,h}}A_{\max,\Lambda,\partial K} h_{K}^{2\min(\ell,\hat{\ell})} (\|c\|_{H^{1+\hat{\ell}}(K)}^2+|K|^\frac{1}{2}\|c\|_{W^{1+\hat{\ell},4}(K)}^2)\right)^\frac{1}{2}
\\
&  C \left(\sum_{\substack{\Lambda\in\mathcal{F}_h^{\text{\rm df-df},0}\\ \Lambda=\partial K_1\cap\partial K_2}} \sum_{i=1}^2 A_\Lambda h_{K_i}^{2\min(\ell,\hat{\ell})} \|c\|_{H^{1+\hat{\ell}}(K_i)}^2+\sum_{\substack{\Lambda\in\mathcal{F}_{\text{\rm df},D,h}\\ \Lambda=\partial K\cap \Gamma_{\text{\rm df},D}}} A_\Lambda h_{K}^{2\min(\ell,\hat{\ell})} \|c\|_{H^{1+\hat{\ell}}(K)}^2\right)^\frac{1}{2}
\\
&+C \left(\sum\limits_{K\in{\cal T}_h} h_K^{2\min(\ell,\hat{\ell})}|K|^\frac{1}{2} {\cal D}_K(A,{\bf u},\gamma,h_K) \|c\|_{W^{1+\hat{\ell},4}(K)}^2\right)^\frac{1}{2}.
\end{aligned}
\end{equation*}
\end{corollary}

Consider the case 
\[
\hat{\ell}=\ell.
\]
An observation is that in order to obtain the optimal convergence order $\ell$ and the optimal SUPG-type convergence order $\ell+\frac{1}{2}$ we must bound the following two terms:
\[
\left(\sum_{\substack{\Lambda\in\mathcal{F}_h^{\text{\rm df-df},0}\cup \mathcal{F}_{\text{\rm df},D,h}\\ \Lambda\subset\partial K}} A_\Lambda h_K^{2\ell} |K|^\frac{1}{2} \|c\|_{W^{1+\ell,4}(K)}^2\right)^\frac{1}{2}
\]
and 
\[
\left(\sum\limits_{K\in{\cal T}_h} h_K^{2\ell}|K|^\frac{1}{2} {\cal D}_K(A,{\bf u},\gamma,h_K) \|c\|_{W^{1+\ell,4}(K)}^2\right)^\frac{1}{2},
\]
while other terms are in essence standard. The following corollary gives the optimal SUPG-type error bound. 

\begin{corollary}\label{cor:supg error bound}
Under the same conditions as in Theorem \ref{thm:error estimates-thm}, for the advection-dominated problem  with
\begin{equation*}
\max(\varepsilon_{\text{\rm df},K},A_{\max,K}^{-1}\sum_{k=1}^\ell \|A\|_{W^{k,\infty}(K)}^2)\le C h_K, 
\end{equation*}
the SUPG-type error bound $O(h^{\ell+\frac{1}{2}})$ holds. 
\end{corollary}
\begin{proof}
From the above observation, we need only bound those two terms. Since $\mathbf{u}\in H(\iv;\Omega)\cap  \prod_{j=1}^J (H^r(D_j))^d$ with $1/2<r\le 1$,  
from the local trace inequality \eqref{eq:local trace L1} and the definition of ${\cal D}_K(A,{\bf u},\gamma,h_K)$ in \eqref{Eq:Stabilizing parameter-1}, we have 
\begin{equation*}
{\cal D}_K(A,{\bf u},\gamma,h_K)\le C h_K |K|^{-\frac{1}{2}} (\|\mathbf{u}\|_{L^2(K)}+h_K^r|\mathbf{u}|_{H^r(K)}+h_K\|\iv\mathbf{u}\|_{L^2(K)}+h_K \|\gamma\|_{L^2(K)}).
\end{equation*}
Then, we have obtained the optimal SUPG-type error bound $\mathcal{O}(h^{\ell+\frac{1}{2}})$, noticing that 
\begin{equation*}
\begin{aligned}
&\left(\sum\limits_{K\in{\cal T}_h} h_K^{2\ell}|K|^\frac{1}{2} {\cal D}_K(A,{\bf u},\gamma,h_K) \|c\|_{W^{1+\ell,4}(K)}^2\right)^\frac{1}{2}\le
\\
&  C \left(\sum\limits_{K\in{\cal T}_h} h_K^{2\ell+1} (\|\mathbf{u}\|_{L^2(K)}+h_K^r |\mathbf{u}|_{H^r(K)}+h_K\|\iv\mathbf{u}\|_{L^2(K)}+h_K \|\gamma\|_{L^2(K)}) \|c\|_{W^{1+\ell,4}(K)}^2\right)^\frac{1}{2}
\\
&\le C h^{\ell+\frac{1}{2}} 
\\
&\times \left(\sum\limits_{K\in{\cal T}_h} (\|\mathbf{u}\|_{L^2(K)}+h_K^r |\mathbf{u}|_{H^r(K)}+h_K\|\iv\mathbf{u}\|_{L^2(K)}+h_K \|\gamma\|_{L^2(K)}) \|c\|_{W^{1+\ell,4}(K)}^2\right)^\frac{1}{2} 
\\
&\le C h^{\ell+\frac{1}{2}}
\\
&\times \left(\sum\limits_{K\in{\cal T}_h} (\|\mathbf{u}\|_{L^2(K)}^2+h_K^{2r} |\mathbf{u}|_{H^r(K)}^2+h_K^2\|\iv\mathbf{u}\|_{L^2(K)}^2+h_K^2 \|\gamma\|_{L^2(K)}^2)\right)^\frac{1}{4} \|c\|_{W^{1+\ell,4}(\Omega)} 
\end{aligned}
\end{equation*}
and that $A_\Lambda h_K^{2\ell}\le C h^{2\ell+1}$ and  
\begin{equation*}
\left(\sum_{K\in\mathcal{T}_h} |K|^\frac{1}{2} \|c\|_{W^{1+\ell,4}(K)}^2\right)^\frac{1}{2}\le |\Omega|^\frac{1}{4}\|c\|_{W^{1+\ell,4}(\Omega)}. 
\end{equation*}
\end{proof}

\begin{remark}\label{rem:DK-1}
In particular, if $\mathbf{u}\in \mathbf{X}_{h}$ and $\gamma=0$, 
so that the following local inverse inequalities hold ($\mathbf{u}$ is still in a finite-dimensional space with respect to $\mathcal{T}_h$ and in any case,    
\begin{equation*}
h_K\|\iv\mathbf{u}\|_{L^2(K)}\le C \|\mathbf{u}\|_{L^2(K)},\quad h_K^r |\mathbf{u}|_{H^r(K)}\le C \|\mathbf{u}\|_{L^2(K)},
\end{equation*}
then  
\begin{equation*}
\begin{aligned}
&\left(\sum\limits_{K\in{\cal T}_h} h_K^{2\ell}|K|^\frac{1}{2} {\cal D}_K(A,{\bf u},\gamma,h_K) \|c\|_{W^{1+\ell,4}(K)}^2\right)^\frac{1}{2}\le C h^{\ell+\frac{1}{2}} \|\mathbf{u}\|_{L^2(\Omega)}^\frac{1}{2} \|c\|_{W^{1+\ell,4}(\Omega)}. 
\end{aligned}
\end{equation*}
\end{remark}

\begin{remark}\label{rem:DK-2}
A computationally more convenient stabilizing parameter, instead of the one \eqref{eq:stabilizing parameter pointwise}, can be defined as the following one: 
 $ 
\tau_K:=\dfrac{\alpha 
h_K^2
}
{{\cal D}_K(A,{\bf u},\gamma,h_K)}$, 
where ${\cal D}_K(A,{\bf u},\gamma,h_K)$ is given by \eqref{Eq:Stabilizing parameter-1}. All the analysis and results are the same applicable, with a slightly more requirement that $\hat{\ell}>\frac{d}{4}$ ($d=2$ or $3$). In particular, when $\hat{\ell}=\ell$, such requirement is satisfied. 
\end{remark}

\subsection*{Conclusion} 

We have designed and analyzed an interior DG method for solving the general steady-state linear partial differential equations (PDEs) of nonnegative characteristic form of mixed Dirichlet and Neumann boundary conditions. Such type of PDEs covers a large class of second-order elliptic-parabolic, first-order hyperbolic equations, with different physical and mathematical natures in the solution over the domain.  We have developed a crucial technique for the design of our DG method and for the purpose of the theoretical analysis:  the set of the interelement boundaries is multiply partitioned into a number of subsets. Such technique is novel. It is the first time such type of PDEs has been comprehensively and correctly studied in a DG method, with correct and minimal averages, jumps and penalties over the set of the diffusion-diffusion interelement boundaries (including the upwind jumps), correctly identifying the continuities and discontinuities of the solution with the help of the technique developed. The correctness of our DG method has been rigorously justified from the consistency result and has been illustrated by the example in \cite{HoustonSchwabSuli2002}. The correctness can be further verified with the two typical cases: (i) The pure diffusion problem, i.e., $\mathbf{u}\equiv 0$ and \eqref{eq:elliptic diffusion} holds; (ii) The pure hyperbolic problem, i.e., $A\equiv 0$. In these two cases, our DG method is the classical DG method, cf. \cite{HoustonSchwabSuli2002}.  In the general case  with \eqref{eq:vanish diffusion}, our DG method is new and novel, not known in the literature. 

We have in addition obtained the optimal error estimates and the optimal SUPG-type error estimates for our DG method for the linearized, steady-state incompressible, miscible displacement problem of nonnegative characteristic form, with the varying, dominating, low regularity advection velocity (Darcy velocity) in $\prod_{j=1}^J (H^r(D_j))^d$. We have introduced a key parameter which is only used for the purpose of the error estimates so that we have managed to obtain the SUPG-type error estimates for the low regularity solutions. It is the first time the SUPG-type error estimates have been obtained under the low regularity advection velocity, where the convergence order is independent of the regularity of the advection velocity. The independence property itself would be very interesting.


\begin{thebibliography}{99}

\bibitem{AyusoMarini2009} B. Ayuso and L. D. Marini. Discontinuous Galerkin methods for advection-diffusion-reaction problems. {\it SIAM J. Numer. Anal.}, 47 (2009), pp. 1391-1420.

\bibitem{BartelsJensenMuller2009} S. Bartels, M. Jensen, and R. M\"{u}ller. Discontinuous Galerkin finite element convergence for incompressible miscible displacement problems of low regularity. {\it SIAM J. Numer. Anal.}, 47(2009), pp.  3720-3743.

\bibitem{BrooksHughes1982} A. N. Brooks and T. J. R. Hughes. Streamline upwind/Petrov-Galerkin formulations for convective dominated flows with particular emphasis on the incompressible Navier-Stokes equations. 
{\it Comput. Methods Appl. Mech. Engrg.},  32 (1982), pp. 199-259.

\bibitem{BuffaHughesSangalli2006} A. Buffa, T. J. R. Hughes, and G. Sangalli. Analysis of a multiscale discontinuous Galerkin method for convection-diffusion problems. {\it SIAM J. Numer. Anal.}, 44 (2006), pp. 1420-1440.

\bibitem{CaiyeZhang2011} Z. Cai, X. Ye, and S. Zhang. Discontinuous Galerkin finite element methods for interface problems: A priori and a posteriori error estimations. {\it SIAM J. Numer. Anal.}, 49(2011), pp. 1761-1787.  

\bibitem{CangianiDongGeorgoulis2022} A. Cangiani, Z. Dong, and E. H. Georgoulis. hp-version discontinuous Galerkin methods on essentially arbitrarilyshaped elements.  {\it Math. Comp.}, 91 (2022), pp. 1-35. 

\bibitem{ChenEwing1999} Z. Chen and R. Ewing. Mathematical analysis for reservoir models. {\it SIAM J. Math. Anal.}, 30 (1999), pp. 431-453.

\bibitem{DiPietroDroniouErn2015} D. A. Di Pietro, J. Droniou, and A. Ern. A discontinuous-skeletal method for advection-diffusion-reaction on general meshes. {\it SIAM J. Numer. Anal.}, 53(2015), pp. 2135-2157. 

\bibitem{DiPietroErnGuermond2008} D. A. Di Pietro, A. Ern, and J. L. Guermond. Discontinuous Galerkin methods for anisotropic semidefinite diffusion with advection. {\it SIAM J. Numer. Anal.}, 46(2008), pp. 805-831.

\bibitem{DroniouEymardPrignetTalbot2019} J. Droniou, R. Eymard, A. Prignet, and K. S. Talbot. Unified convergence analysis of numerical schemes for a miscible displacement problem. {\it Found. Comput. Math.},  19(2019), pp. 333-374. 

\bibitem{DuanMa2024} H. Y. Duan and J. H. Ma. A penalty-free and essentially stabilization-free DG method for convection-dominated second-order elliptic problems. {\it J. Sci. Comput.}, (2024) 100: 59.

\bibitem{Feng1995} X. Feng. On existence and uniqueness results for a coupled system modeling miscible displacement in porous media. {\it J. Math. Anal. Appl.}, 194 (1995), pp. 883-910.

\bibitem{FrancaValentine2000} L. P. Franca and F. Valentin. On an improved unusual stabilized finite element method for the advective-reactive-diffusive equation. {\it Comput. Methods Appl. Mech. Engrg.}, 190 (2000), pp. 1785-1800.

\bibitem{GaoSun2020} H. Gao and W. Sun. Optimal error analysis of Crank-Nicolson lowest-order Galerkin-mixed finite element method for incompressible miscible flow in porous media. {\it Numer. Methods Partial Differential Eq.},  36(2020), pp. 1773–1789. 

\bibitem{GiraultLiRiviere2016}  V. Girault, J. Li, and B. M. Rivi\'{e}re. Strong convergence of the discontinuous Galerkin scheme for the low regularity miscible displacement equations. {\it Numer. Methods Partial Differential Eq.},  33(2017), pp. 489–513.

\bibitem{HoustonSuli2001} P. Houston and E, Suli. Stabilized hp-finite element approximation of partial differential equations with non-negative characteristic form. {\it Computing}, 66(2001), pp. 99-119.

\bibitem{HoustonSchwabSuli2002} P. Houston, C. Schwab, and E. S\"{u}li. Discontinuous hp-finite element methods for advection-diffusion-reaction problems. {\it SIAM J. Numer. Anal.}, 39 (2002), pp. 2133-2163.

\bibitem{SunKou2015} J. Kou and S. Sun. Analysis of a combined mixed finite element and discontinuous Galerkin method for incompressible two-phase flow in porous media. {\it Math. Meth. Appl. Sci.},  37(2014), pp. 962–982. 
   
\bibitem{LiRiviereWalkington2015} J. Li, B. M. Rivi\`{e}re, and N. J. Walkington. Convergence of a high order method in time and space for the miscible displacement equations. {\it ESAIM: M2AN},  49(2015), pp. 953-976. 
  
\bibitem{LiSun2013} B. Li and W. Sun. Unconditional convergence and optimal error estimates of a Galerkin-mixed FEM for incompressible miscible flow in porous media. {\it SIAM J. Numer. Anal.}, 51 (2013), pp. 1959–1977.

\bibitem{ProftRiviere2009} J. Proft and B. M. Rivi\`{e}re. Discontinuous Galerkin methods for convective-diffusive equations for varing and vanishing diffusivity. {\it Int. J.  Numer. Anal. Model.}, 6(2009), pp. 533-561.

\bibitem{RiviereWalkington2011} B. M. Rivi\`{e}re and N. J. Walkington. Convergence of a discontinuous Galerkin method for the miscible displacement equation under low regularity. {\it SIAM J. Numer. Anal.},  49(2011), pp. 1085-1110.

\bibitem{SunYuan2009} T. Sun and Y. Yuan. An approximation of incompressible miscible displacement in porous media by mixed finite element method and characteristics-mixed finite element method. {\it J. Comput. Appl. Math.}, 228 (2009), pp. 391-411.

\bibitem{SunWu2021} W. Sun and C. Wu. New analysis of Galerkin-mixed FEMs for incompressible miscible flow in porous media. {\it Math. Comp.},  90 (2021), pp. 81-102.

\bibitem{Sun2021} W. Sun. Analysis of lowest-order characteristics-mixed FEMs for incompressible miscible flow in porous media. {\it SIAM J. Numer. Anal.}, 59 (2021), pp. 1875-1895. 

\bibitem{Wang2008} H. Wang. An optimal-order error estimate for a family of ELLAM-MFEM approximations to porous medium flow. {\it SIAM J. Numer. Anal.}, 46(2008), pp.2133–2152.

\bibitem{WangSiSun2014} J. Wang, Z. Si, and W. Sun. A new error analysis of characteristics-mixed FEMs for miscible displacement in porous media. {\it SIAM J. Numer. Anal.}, 52 (2014), pp. 3000-3020. 

\bibitem{WangZhengYuGuoZhang2019} H. Wang, J. Zheng, F. Yu, H. Guo, and Q. Zhang. Local discontinuous Galerkin method with implicit-explicit time marching for incompressible miscible displacement problem in porous media. {\it J. Sci. Comput.},  78(2019), pp. 1-28. 

\end{thebibliography}
\end{document}